\documentclass{amsart}

\usepackage{amsthm,amsxtra}
\usepackage{amssymb}
\usepackage{amsmath}
\usepackage{amsfonts}

\usepackage{graphicx}
\usepackage{xcolor}

\numberwithin{equation}{section}
\allowdisplaybreaks

\theoremstyle{plain}
\newtheorem{theorem}{Theorem}[section]
\newtheorem{proposition}[theorem]{Proposition}
\newtheorem{lemma}[theorem]{Lemma}
\newtheorem{corollary}[theorem]{Corollary}

\theoremstyle{definition}
\newtheorem{definition}[theorem]{Definition}
\newtheorem{example}[theorem]{Example}

\theoremstyle{remark}
\newtheorem{remark}[theorem]{Remark}

\def\to{\rightarrow}

\def\dim{\operatorname{dim}}

\def\Gr{\operatorname{Gr}}
\def\SL{\operatorname{SL}}
\def\GL{\operatorname{GL}}

\def\SO{\operatorname{SO}}

\def\Spin{\operatorname{Spin}}

\def\Sp{\operatorname{Sp}}

\def\Sing{\operatorname{Sing}}

\begin{document}

\title[Equivariant Ulrich bundles on exceptional homogeneous varieties]{Equivariant Ulrich bundles on exceptional homogeneous varieties}

\author{Kyoung-Seog Lee and Kyeong-Dong Park}

\address{Kyoung-Seog Lee \\ Center for Geometry and Physics \\ Institute for Basic Science (IBS) \\ Pohang 37673, Korea
\vskip 0.2em 
Current Address: 
Institute of the Mathematical Sciences of the Americas, University of Miami, 1365 Memorial Drive, Ungar 515, Coral Gables, FL 33146, USA}
\email{kyoungseog02@gmail.com}

\address{Kyeong-Dong Park \\ Center for Geometry and Physics \\ Institute for Basic Science (IBS) \\ Pohang 37673, Korea
\vskip 0.2em 
Current Address: 
School of Mathematics, Korea Institute for Advanced Study (KIAS), Dongdaemun-gu, Seoul 02455, Republic of Korea}
\email{kdpark@kias.re.kr}

\thanks{The authors were supported by the Institute for Basic Science (IBS-R003-Y1).}



\begin{abstract} 
We prove that the only rational homogeneous varieties with Picard number 1 of the exceptional algebraic groups admitting irreducible equivariant Ulrich vector bundles 
are the Cayley plane $E_6/P_1$ and the $E_7$-adjoint variety $E_7/P_1$.
From this result, we see that a general hyperplane section $F_4/P_4$ of the Cayley plane also has an equivariant but non-irreducible Ulrich bundle.
\end{abstract}

\maketitle

\section{Introduction}

Vector bundles on algebraic varieties are fundamental objects to study in order to understand geometry and topology of the varieties. 
Topological type of a vector bundle on an algebraic curve is determined by rank and degree of the vector bundle, and there are lots of studies about 
vector bundles of given rank and degree on curves. 
For higher-dimensional algebraic varieties, 
we need more data to determine topological types of vector bundles on them 
and it makes hard to decide which vector bundles are most interesting objects to study on higher-dimensional algebraic varieties.

Ulrich bundles are vector bundles which enjoy many special features, 
and existence and properties of Ulrich bundles on a given algebraic variety tell us many properties of the variety. 
Therefore Ulrich bundles form natural candidates of vector bundles on higher-dimensional algebraic varieties to investigate. 
Eisenbud and Schreyer asked whether every projective variety admits an Ulrich sheaf in \cite{ES} 
and their question has been answered positively 
for several algebraic varieties, 
e.g., complete intersection varieties \cite{HUB}, del Pezzo surfaces \cite{ES}, K3 surfaces \cite{AFO}, abelian surfaces \cite{Be16}, and non-special surfaces with $p_g= q= 0$ \cite{Cas}. 
Recently moduli spaces of Ulrich bundles were studied for some algebraic varieties, 
and it seems that they enjoy many interesting properties and reflect many geometric features of the given varieties (cf. \cite{LMS1, LMS2, CKL, LP}).

Equivariant Ulrich bundles on some rational homogeneous varieties were studied by many authors. 
To be more precise, Costa and Mir\'{o}-Roig \cite{CMR1}, and Coskun et al. \cite{CCHMW} 
classified irreducible equivariant Ulrich bundles on Grassmannians, and partial flag varieties of algebraic groups of type $A$, respectively. 
Then Fonarev \cite{Fo} classified irreducible equivariant Ulrich bundles on isotropic Grassmannians of algebraic groups associated to other classical groups of type $B, C, D$. 
Being motivated by the above works, we classify irreducible equivariant Ulrich bundles on 
rational homogeneous varieties with Picard number 1 of the exceptional algebraic groups in this paper. 

Let us recall some basic definitions and properties of algebraic groups to state the result more precisely.
Let $G$ be a simple algebraic group 
over $\mathbb C$.
We fix a maximal torus $T \subset G$. 
Let $\Delta=\{ \alpha_1, \cdots, \alpha_n\}$ be the system of simple roots of $G$  
and $P=P_k$ denote the maximal parabolic subgroup associated to a simple root $\alpha_k$
following the standard numbering (e.g., 
\cite{Humphreys}). 
We know that the category of $G$-equivariant vector bundles on $G/P$ is equivalent to 
the category of finite-dimensional representations of $P$. Since $P$ has a Levi decomposition $P=LU$, 
where $U$ is the unipotent radical of $P$ and $L$ is a Levi factor, and $U$ acts trivially, 
irreducible representations of $P$ are completely determined by representations of the reductive group $L$. 

The weight lattice of $L$ is canonically isomorphic to the weight lattice of $G$. 
Let $\Lambda^+_L$ denote the cone of integral $L$-dominant weights, 
which is generated by the fundamental weights $\omega_1, \cdots, \omega_n$ and $-\omega_k$. 
Given $\omega \in \Lambda_L^+$, we have an irreducible representation $V(\omega)$ of $P$ with 
highest weight $\omega$.
We denote by $\mathcal E_{\omega}$ the corresponding irreducible equivariant vector bundle $G\times_P V(\omega)^*$ on $G/P$, 
where $V(\omega)^*$ is the dual representation of $V(\omega)$.

In this paper, we classify irreducible equivariant Ulrich bundles on rational homogeneous varieties with Picard number 1 of the exceptional algebraic groups. 
In particular, we prove that three more homogeneous varieties admit equivariant Ulrich bundles on them.

\begin{theorem} 
The only rational homogeneous varieties with Picard number 1 of the exceptional algebraic groups admitting an irreducible equivariant Ulrich bundle 
are the Cayley plane $E_6/P_1 \cong E_6/P_6$ and the $E_7$-adjoint variety $E_7/P_1$.
Furthermore, the irreducible equivariant bundles $\mathcal E_{\omega_5+3\omega_6}$ on $E_6/P_1$ and $\mathcal E_{\omega_5+3\omega_6+8\omega_7}$ on $E_7/P_1$ are Ulrich, 
and they are the only irreducible equivariant Ulrich bundle up to isomorphism, respectively. 
\end{theorem}

It is well-known that the homogeneous variety $F_4/P_4$ is isomorphic to a general hyperplane section of the Cayley plane $E_6/P_1$. 
Hence we obtain the following corollary.

\begin{corollary}
Let $\mathcal{E}_{\omega_5+3\omega_6}$ be the irreducible equivariant Ulrich bundle on the Cayley plane $E_6/P_1$ in the above theorem. 
Then the restriction $\mathcal{E}_{\omega_5+3\omega_6}|_{F_4/P_4}$ is an equivariant but non-irreducible Ulrich bundle on the rational homogeneous variety $F_4/P_4$.
\end{corollary}

It turns out that most of exceptional homogeneous varieties do not admit an irreducible equivariant Ulrich bundle according to the negative result obtained in this paper. 
However, from the above corollary,
we see that there can be equivariant but non-irreducible Ulrich bundles on the other homogeneous varieties of exceptional groups. 
For example, $G_2/P_1$ is isomorphic to the 5-dimensional quadric so it also has an equivariant but non-irreducible Ulrich bundle. 
As also pointed out in Remark 6.7 of \cite{Fo}, it will be an interesting task to find non-irreducible equivariant Ulrich bundles on homogeneous varieties.

\section{
Ulrich bundles and the Borel--Weil--Bott theorem}

There are several equivalent definitions of Ulrich bundles using linear resolutions or cohomologies.
We are going to use the cohomological characterization.

\begin{definition}
Let $X\subset \mathbb P^N$ be a smooth projective variety of dimension $d$ over $\mathbb C$. 
A vector bundle $E$ on $X$ is called \emph{Ulrich} if the cohomology groups $H^i(X, E(-t))=0$ for all $0 \leq i \leq d$ and $1 \leq t \leq d$. 
Here, we denote the twisted bundle $E \otimes \mathcal O_X(-t)$ by $E(-t)$.
\end{definition}

More geometrically, an Ulrich bundle can be characterized as a \emph{trivial} one in the sense: 
the push-forward $\pi_* E$ is a trivial bundle for a general linear projection $\pi \colon X \to \mathbb P^d$; 
see \cite{ES} and Theorem 2.3 of \cite{Be17} for the proof. 

\begin{definition}
A vector bundle $E$ on a rational homogeneous varietiy $G/P$ is \emph{equivariant} 
if the action of $G$ on the base space $G/P$ lifts to a compatible action of $G$ on $E$ via bundle automorphisms.
\end{definition}

For an integral weight $\omega$ dominant with respect to $P$, 
we have an irreducible representation $V(\omega)$ of $P$ with 
highest weight $\omega$, and 
denote by $\mathcal E_{\omega}$ the corresponding \emph{irreducible equivariant vector bundle} $G\times_P V(\omega)^*$ on $G/P$:
$$\mathcal E_{\omega}:= G \times_P V(\omega)^* = (G \times V(\omega)^*)/P,$$
where the equivalence relation is given by $(g, v) \sim (gp, p^{-1} . v)$ for $p \in P$.
To compute the cohomology of equivariant vector bundles on a rational homogeneous variety $G/P$, 
we use the famous Borel--Weil--Bott theorem from \cite{Bo57}. 
For an introduction to this theorem, we refer \cite{Bo88, Snow, W} and Chapter 5 of \cite{BE}.

\begin{theorem}[Borel--Weil--Bott theorem] 
Let $G$ be a simply connected complex semisimple algebraic group and $P \subset G$ a parabolic subgroup.
Suppose that $\omega$ is an integral weight for $G$ dominant with respect to $P$. Let $\rho$ denote the sum of fundamental weights of $G$.
\begin{itemize}
\item If a weight $\omega+\rho$ is singular, that is, it is orthogonal to some (positive) root of $G$, equivalently, it lies on a wall of a Weyl chamber, 
then all cohomology groups $H^i(G/P, \mathcal E_{\omega})$ vanish. 
\item Otherwise, $\omega+\rho$ is regular, that is, it lies in the interior of some Weyl chamber;  
then $H^{\ell(w)}(G/P, \mathcal E_{\omega})=V_G(w(\omega+\rho)-\rho)^*$ and any other cohomology vanishes. 
Here, $w\in W$ is the unique element of the Weyl group of $G$ such that $w(\omega+\rho)$ is strictly dominant, 
and $\ell(w)$ means the length of $w \in W$, that is, the minimal integer $\ell(w)$ such that $w$ can be expressed as a product of $\ell(w)$ simple reflections.  
\end{itemize}
\end{theorem}

From now on, we will consider only a rational homogeneous variety $G/P$ with Picard number 1 
for the maximal parabolic subgroup $P$ associated to a simple root $\alpha_k \in \Delta$ of $G$.
We always realize $G/P$ as a complex projective variety by giving a 
minimal equivariant embedding into the projective space of a finite-dimensional $G$-module.
Since the weight $\omega_k$ defines a character of $L \subset P_k$, the associated equivariant bundle $\mathcal E_{\omega_k}$ on $G/P_k$ is just a line bundle.
In fact, the line bundle $\mathcal E_{\omega_k}$ on $G/P_k$ is the positive generator of its Picard group and gives the minimal equivariant embedding 
$G/P_k \cong G.[v] \subset \mathbb P(V_G(\omega_k))$, 
where $v$ is a highest weight vector in the irreducible $G$-representation $V_G(\omega_k)$ with highest weight $\omega_k$; see e.g. Chapter 6 of \cite{BE}.

Since conjugacy classes of parabolic subgroups of a simple algebraic group are in one-to-one correspondence with subsets of
the set of simple roots (equivalently, nodes of the corresponding Dynkin diagram $\mathcal D(G)$),
the Dynkin diagrams $(\mathcal D(G), \alpha_k)$ with a marked node correspond to rational homogeneous varieties with Picard number 1.
Because the set $\{\omega_i \}_{1 \leq i \leq n}$ of fundamental weights is a basis for the dual Cartan subalgebra $\mathfrak t^*$, 
any weight can be expressed a linear combination of fundamental weights. 
Then we represent an integral weight and the corresponding equivariant vector bundle by inscribing the coefficients in this expression over the nodes of the Dynkin diagram for $G$.

\begin{example} \label{Gr and Q5}
\begin{enumerate}
\item The Grassmannian $\Gr(2,4)$ is isomorphic to $\SL(4, \mathbb C)/ P_2 
$ as homogeneous spaces, 
so that we denote it by the marked Dynkin diagram $(A_3, \alpha_2)$ with a crossed node corresponding to the simple root $\alpha_2$.
Consider the irreducible equivariant vector bundle 
$\mathcal E_{\omega_1}$ associated to the first fundamental weight $\omega_1$, 
which is the dual $\mathcal S^*$ of the (rank 2) universal subbundle on $\Gr(2,4)$.
Since $\mathcal E_{\omega_2} = \mathcal O(1)$, the weight defining $\mathcal E_{\omega_1}(-t)$ is $\omega_1 - t \, \omega_2$. 
Then $\omega_1 - t \, \omega_2 + \rho = 2\omega_1 + (1-t) \omega_2 + \omega_3$. 

To compute the action of the Weyl group $W$ on the weight lattice, 
we need an explicit formula for the action of a simple reflection $\sigma_i$ on a weight $\lambda$ represented in the Dynkin diagram notation.
From $\sigma_i(\lambda)=\lambda-2\frac{(\lambda, \alpha_i)}{(\alpha_i, \alpha_i)}\alpha_i=\lambda - (\lambda, \alpha_i^{\vee})\alpha_i$, 
where $(\, , \,)$ is the Cartan-Killing form on $\mathfrak t^*$ and $\alpha^{\vee}$ is the coroot of $\alpha$, 
the node coefficients for $\sigma_i(\lambda)$ are given by 
$(\sigma_i(\lambda), \alpha_j^{\vee}) = (\lambda, \alpha_j^{\vee}) - (\lambda, \alpha_i^{\vee}) (\alpha_i, \alpha_j^{\vee})$.
Because $(\alpha_i, \alpha_j^{\vee})$ are \emph{Cartan integers $c_{ij}$} which specify the semisimple Lie algebra $\mathfrak g$, 
this equation yields the following: 
Let $c$ be the coefficient of the node corresponding to $\alpha_i$. 
In order to compute $\sigma_i(\lambda)$, add $c$ to the adjacent coefficients with multiplicity, if there is a multiple edge directed towards the adjacent node, 
and then replace $c$ by $-c$. 

\vskip 0.5em

 \begin{picture} (250, 35)
 \put(0,20){\circle{5}} 
 \put(25.8,18.2){$\times$} 
 \put(61,20){\circle{5}} 

 \put(2.5,20){\line(1,0){25}} 
 \put(32.5,20){\line(1,0){25}} 

 \put(-3,28){$2$} 
 \put(18,28){$1-t$} 
 \put(58,28){$1$} 

 \put(78,20){$\longrightarrow$} 
 \put(83,13){$\sigma_2$} 

 \put(110,20){\circle{5}} 
 \put(135.8,18.2){$\times$} 
 \put(171,20){\circle{5}} 

 \put(112.5,20){\line(1,0){25}} 
 \put(142.5,20){\line(1,0){25}} 

 \put(98,28){$3-t$} 
 \put(122,8){$-1+t$} 
 \put(160,28){$2-t$} 

 \put(190,20){$\longrightarrow$} 
 \put(195,13){$\sigma_1$} 

 \put(230,20){\circle{5}} 
 \put(255.8,18.2){$\times$} 
 \put(291,20){\circle{5}} 

 \put(232.5,20){\line(1,0){25}} 
 \put(262.5,20){\line(1,0){25}} 

 \put(213,28){$-3+t$} 
 \put(258,28){$2$} 
 \put(280,28){$2-t$} 

 \put(300,20){$\longrightarrow$} 
 \put(305,13){$\sigma_3$} 

 \put(330,20){\circle{5}} 
 \put(355.8,18.2){$\times$} 
 \put(391,20){\circle{5}} 

 \put(332.5,20){\line(1,0){25}} 
 \put(362.5,20){\line(1,0){25}} 

 \put(318,28){$-3+t$} 
 \put(348,8){$4-t$} 
 \put(373,28){$-2+t$} 
 \end{picture}
\newline
After adding $\rho$, we apply simple reflections on the weight $\omega_1 - t \omega_2 + \rho$ obtaining singular weights for $t=1,2,3,4$.  
By the Borel-Weil-Bott theorem, we know $H^i(\Gr(2,4), \mathcal E_{\omega_1}(-t))=0$ for all $0\leq i \leq 4$ and $1\leq t \leq 4$. 
Hence, the irreducible equivariant vector bundle $\mathcal E_{\omega_1}$ on $\Gr(2,4)$ is Ulrich.
Similarily, we know that the (rank 2) universal quotient bundle $\mathcal E_{\omega_3}$ on $\Gr(2,4)$ is also Ulrich.

\item For a special orthogonal group $B_3=\SO(7, \mathbb C)$, 
a homogeneous space $B_3/P_1$ is a variety of isotropic 1-dimensional subspaces with respect to a nondegenerate symmetric bilinear form on $\mathbb C^7$. 
Hence, under the minimal equivariant embedding, this variety is a quadric hypersurface $\mathbb Q^5$ in $\mathbb P^6$. 

Since a Levi factor $L$ of the parabolic subgroup $P_1$ is isomorphic to $\mathbb C^* \times \SO(5, \mathbb C)$ 
and $V_{P_1}(\omega_3)$ is the 4-dimensional spin representaion of $\SO(5, \mathbb C)$, 
the irreducible equivariant vector bundle $\mathcal E_{\omega_3}$ associated to the third fundamental weight $\omega_3$
is isomorphic to the (rank 4) spinor bundle on $\mathbb Q^5$.
Since $\mathcal E_{\omega_1} = \mathcal O(1)$, 
the weight defining $\mathcal E_{\omega_3}(-t)$ is $-t \, \omega_1 + \omega_3$. 
After adding $\rho$, we apply simple reflections on the weight $-t \, \omega_1 + \omega_3 + \rho = (1-t) \omega_1 + \omega_2 + 2\omega_3$. 

\vskip 0.5em

 \begin{picture} (250, 40)
 \put(-4,18){$\times$} 
 \put(30,20){\circle{5}} 
 \put(60,20){\circle{5}} 

 \put(2.5,20){\line(1,0){25}} 
 \put(32.5,18.5){\line(1,0){25}} 
 \put(32.5,21.5){\line(1,0){25}} 

 \put(43,17){$\big >$} 

 \put(-13,28){$1-t$} 
 \put(28,28){$1$} 
 \put(58,28){$2$} 

 \put(78,20){$\longrightarrow$} 
 \put(83,13){$\sigma_1$} 

 \put(106,18){$\times$} 
 \put(140,20){\circle{5}} 
 \put(170,20){\circle{5}} 

 \put(112.5,20){\line(1,0){25}} 
 \put(142.5,18.5){\line(1,0){25}} 
 \put(142.5,21.5){\line(1,0){25}} 

 \put(153,17){$\big >$} 

 \put(95,28){$t-1$} 
 \put(130,28){$2-t$} 
 \put(168,28){$2$} 

 \put(190,20){$\longrightarrow$} 
 \put(195,13){$\sigma_2$} 

 \put(226,18){$\times$} 
 \put(260,20){\circle{5}} 
 \put(290,20){\circle{5}} 

 \put(232.5,20){\line(1,0){25}} 
 \put(262.5,18.5){\line(1,0){25}} 
 \put(262.5,21.5){\line(1,0){25}} 

 \put(273,17){$\big >$} 

 \put(227,28){$1$} 
 \put(248,28){$t-2$} 
 \put(285,28){$6-2t$} 
 \end{picture}

 \begin{picture} (250, 40)
 \put(78,20){$\longrightarrow$} 
 \put(83,13){$\sigma_3$} 

 \put(106,18){$\times$} 
 \put(140,20){\circle{5}} 
 \put(170,20){\circle{5}} 

 \put(112.5,20){\line(1,0){25}} 
 \put(142.5,18.5){\line(1,0){25}} 
 \put(142.5,21.5){\line(1,0){25}} 

 \put(153,17){$\big >$} 

 \put(105,28){$1$} 
 \put(128,28){$4-t$} 
 \put(162,28){$2t-6$} 

 \put(190,20){$\longrightarrow$} 
 \put(195,13){$\sigma_2$} 

 \put(226,18){$\times$} 
 \put(260,20){\circle{5}} 
 \put(290,20){\circle{5}} 

 \put(232.5,20){\line(1,0){25}} 
 \put(262.5,18.5){\line(1,0){25}} 
 \put(262.5,21.5){\line(1,0){25}} 

 \put(273,17){$\big >$} 

 \put(215,28){$5-t$} 
 \put(249,28){$t-4$} 
 \put(288,28){$2$} 
 \end{picture}
\newline
Because the weight $-t \, \omega_1 + \omega_3 + \rho$ are singular for $t=1, 2, 3, 4, 5$, 
we know $H^i(B_3/P_1, \mathcal E_{\omega_3}(-t))=0$ for all $0\leq i \leq 5$ and $1\leq t \leq 5$ by the Borel-Weil-Bott theorem. 
Hence, the spinor bundle $\mathcal E_{\omega_3}$ on $B_3/P_1 \cong \mathbb Q^5$ is Ulrich.
In fact, up to isomorphism the spinor bundle is the only irreducible $B_n$-equivariant Ulrich bundle on an odd-dimensional quadric $\mathbb Q^{2n-1}$; 
see Corollary 4.3 of \cite{Fo}. 
\end{enumerate}
\end{example}

\begin{remark}
Let $\mathbb Q^n \subset \mathbb P^{n+1}$ be a smooth quadric. 
If $n$ is odd, then $\mathbb Q^n \cong B_{\frac{n+1}{2}}/P_1$ 
and the spinor bundle $\Sigma=\mathcal E_{\omega_{(n+1)/2}}$
of rank $2^{\frac{n-1}{2}}$ is a unique indecomposable Ulrich bundle on $\mathbb Q^n$; see Proposition 2.5 of \cite{Be17}. 
In fact, the vector bundles $\langle \Sigma, \mathcal O(1), \cdots , \mathcal O(n)\rangle$ form a semi-orthogonal decomposition of the bounded derived category $D^b(\mathbb Q^n)$ of coherent sheaves on $\mathbb Q^n$; see \cite{Kap}. 
Likewise, if $n$ is even, then $\mathbb Q^n \cong D_{\frac{n}{2}+1}/P_1$ and 
$\mathbb Q^n$ admits exactly two indecomposable Ulrich bundles $\Sigma_+, \Sigma_-$ which are non-isomorphic spinor bundles of rank $2^{\frac{n}{2}-1}$. 
The vector bundles $\langle \Sigma_-, \Sigma_+, \mathcal O(1), \cdots , \mathcal O(n)\rangle$ form a semi-orthogonal decomposition of $D^b(\mathbb Q^n)$. 
For instance, since the Grassmannian $\Gr(2,4)$ is isomorphic to $\mathbb Q^4$, 
$\Sigma_+, \Sigma_-$ are the dual universal subbundle $\mathcal E_{\omega_1}$ and the universal quotient bundle $\mathcal E_{\omega_3}$ considered in Example \ref{Gr and Q5} (1). 
\end{remark}

\section{Criterion for irreducible equivariant Ulrich bundles }

From the Borel--Weil--Bott theorem, 
we can obtain a simple criterion for an irreducible equivariant vector bundle on a rational homogeneous variety $G/P$ to be Ulrich. 

\begin{definition} 
Let $L$ be a Levi factor of a maximal parabolic subgroup $P_k \subset G$. 
For an $L$-dominant integral weight $\omega \in \Lambda_L^+$, we define a set 
$$\Sing(\omega):=\{ t \in \mathbb Z : \omega + \rho - t \, \omega_k \mbox{ is singular}\},$$ 
where $\rho$ is the sum of fundamental weights of $G$.
\end{definition}

Denote the set of positive roots of $G$ by $\Phi^+$. 
Fix $k \in \{1, 2, \cdots, n \}$ and the parabolic subgroup $P$ associated to $\alpha_k \in \Delta$. 
Given an integer $m$, we define $\Phi_m$ as the set of all roots $\sum_{i=1}^n c_i \alpha_i$ with $c_k = m$. 
Then the maximal parabolic subalgebra associated to a simple root $\alpha_k$ has the decomposition 
$\mathfrak p = \mathfrak t \oplus \bigoplus_{\alpha \in \Phi_0} \mathfrak g_{\alpha} \oplus \bigoplus_{\alpha \in \Phi^+ \setminus \Phi_0} \mathfrak g_{\alpha}$
and we can identify the tangent space $T_o(G/P)$ at a base point $o$ of a rational homogeneous variety $G/P$ with $\bigoplus_{\alpha \in \Phi^+ \setminus \Phi_0} \mathfrak g_{-\alpha}$. 
In particular, the dimension of $G/P$ is equal to the cardinality $| \Phi^+ \setminus \Phi_0 |$. 

If $\omega = \sum_{i=1}^n a_i \omega_i$ is an $L$-dominant weight, then $a_i \geq 0$ for all $i \neq k$. 
For $\alpha \in \Phi_0 \cap \Phi^+$, 
we know that $(\omega_k , \alpha^{\vee})=0$ from the relation $(\omega_i, \alpha_j^{\vee})=\delta_{ij}$, 
which implies that $(\omega + \rho - t \, \omega_k , \alpha^{\vee}) = (\sum_{i=1}^n (a_i +1) \omega_i - t \, \omega_k , \alpha^{\vee}) = \sum_{i \neq k} (a_i +1)(\omega_i , \alpha^{\vee}) >0$ 
for any $\omega \in \Lambda_L^+$ and $t \in \mathbb Z$. 
Hence, one only need to consider positive roots in $\Phi^+ \setminus \Phi_0$ for the set $\Sing(\omega)$ and 
the cardinality of $\Sing(\omega)$ is at most the dimension of $G/P_k$. 


\begin{lemma}[Fonarev's criterion, Lemma 2.4 of \cite{Fo}] \label{Ulrich criterion}
For $\omega \in \Lambda_L^+$, 
an irreducible equivariant vector bundle $\mathcal E_{\omega}$ on a rational homogeneous variety $G/P_k$ with Picard number 1 and dimension $d$ is Ulrich if and only if 
$$\Sing(\omega) = \{ 1, 2, \cdots, d-1, d \}.$$ 
\end{lemma}

\begin{proof}
Assume that $\Sing(\omega) = \{ 1, 2, \cdots, d-1, d \}$. 
Because $\mathcal E_{\omega}(-t) \cong \mathcal E_{\omega - t \omega_k}$ and $\omega + \rho - t \, \omega_k$ is singular for any $t \in \Sing(\omega)$, 
it follows from the Borel--Weil--Bott theorem that $H^{\bullet}(G/P_k, \mathcal E_{\omega}(-t) )=0$ for $1 \leq t \leq d$.  
Thus, the vector bundle $\mathcal E_{\omega}$ is Ulrich.

Next, if the bundle $\mathcal E_{\omega}$ is Ulrich, then the set $\{ 1, 2, \cdots, d-1, d \}$ is contained in $\Sing(\omega)$ by the Borel--Weil--Bott theorem. 
Since $| \Sing(\omega) | \leq d$, we complete the proof.  
\end{proof}

The above Lemma is our key tool to study (non-)existence 
of irreducible equivariant Ulrich bundles on rational homogeneous varieties. 
We compute the set of singular values 
and analyze $\Sing(\omega)$ case by case. 

\begin{proposition} \label{coefficient of weight}
For $\omega = \sum_{i=1}^n a_i \omega_i \in \Lambda_L^+$, 
if an irreducible equivariant vector bundle $\mathcal E_{\omega}$ on $G/P_k$ is Ulrich, then the $k$-th coefficient $a_k$ of $\omega$ must be zero.
\end{proposition}

\begin{proof}
We write a positive root $\alpha \in \Phi^+ \setminus \Phi_0$ as $\alpha = \sum_{i=1}^n c_i \alpha_i$ with $c_k > 0$ and $c_i \geq 0$ for all $i \neq k$. 
Then we have 
\begin{eqnarray*}
(\omega + \rho - t \omega_k , \alpha^{\vee}) 
&=& \Big(\sum_{i \neq k} (a_i +1) \omega_i + (a_k +1 - t) \omega_k , \frac{1}{(\alpha, \alpha)} \sum_{j=1}^n c_j (\alpha_j, \alpha_j)\alpha_j^{\vee} \Big) \\
&=& \frac{1}{(\alpha, \alpha)} \Big\{ \sum_{i \neq k} c_i (\alpha_i, \alpha_i) (a_i +1) + c_k (\alpha_k, \alpha_k) (a_k +1 - t) \Big\}. 
\end{eqnarray*}
If the weight $\omega + \rho - t \omega_k$ is orthogonal to $\alpha$, then we obtain a singular value 
\begin{equation}\label{singular values}
t = (a_k +1) + \sum_{i \neq k} \frac{c_i (\alpha_i, \alpha_i) (a_i +1)}{c_k (\alpha_k, \alpha_k) }.
\end{equation}
So we deduce that $\mbox{min} \Sing(\omega) = a_k + 1$, where the minimum is attained when $\alpha = \alpha_k$. 
Since $\mbox{min} \Sing(\omega) = 1$ by Lemma \ref{Ulrich criterion}, we obtain $a_k = 0$.
\end{proof}

\begin{proposition} \label{subminimum}
Let the Dynkin diagram of $G$ be simply laced, that is, of type $A, D, E$. 
For $\omega = \sum_{i=1}^n a_i \omega_i \in \Lambda_L^+$, assume that an irreducible equivariant vector bundle $\mathcal E_{\omega}$ on $G/P_k$ is Ulrich.
\begin{enumerate}
\item If $\alpha_k$ is an extremal node of the Dynkin diagram of $G$, then the coefficient of $\omega$ corresonding to the node adjacent to $\alpha_k$ must be zero.
\item If $\alpha_k$ is not an extremal node of the Dynkin diagram of $G$, then at least one of the coefficients of $\omega$ corresonding to the nodes adjacent to $\alpha_k$ must be zero.
\end{enumerate}
\end{proposition}

\begin{proof}
Because the lengths of all simple roots are the same when $G$ is of type $A, D, E$, 
the singular value corresponding to a root $\alpha = \sum_{i=1}^n c_i \alpha_i$ in $\Phi^+ \setminus \Phi_0$ 
is given by $t = 1 + \sum_{i \neq k} \frac{c_i (a_i +1)}{c_k}$ 
from Equation (\ref{singular values}) above. 

(i) When $G$ is of type $A$, the highest root of $G$ is $\alpha_1 + \cdots + \alpha_n$. 
If $\alpha_k$ is not an extremal node of the Dynkin diagram of $G$, we know that $t = 1 + \sum_{i \neq k} c_i (a_i +1) >2$ except for $\alpha =  \alpha_{k-1} + \alpha_{k}$ and $\alpha_k + \alpha_{k+1}$. 
By Lemma~\ref{Ulrich criterion}, either $1+(a_{k-1}+1)$ or $1+(a_{k+1}+1)$ should be equal to 2, which implies that one of either $a_{k-1}$ or $a_{k+1}$ must be zero. 
Similarly, if $\alpha_k$ is an extremal node of the Dynkin diagram of $G$, then the coefficient of $\omega$ corresonding to the node adjacent to $\alpha_k$ must be zero. 

(ii) When $G$ is of type $D$, the highest root of $G$ is $\alpha_1 + 2\alpha_2 + \cdots + 2\alpha_{n-2} + \alpha_{n-1} + \alpha_n$. 
If $\alpha_k$ is an extremal node, then the result follows from the same argument as (i). 
If $n\geq 4$ and $\alpha_k$ is not an extremal node, 
then $t = 1 + \sum_{i \neq k} \frac{c_i (a_i +1)}{2}>2$ for all $\alpha \in \Phi_2$ 
since $\alpha \in \Phi_2$ is expressed as the sum of at least 4 simple roots. 
Thus, when $\alpha_{\ell}$ correspondes to a node adjacent to $\alpha_k$ in the Dynkin diagram of $G$, at least one of $1+(a_{\ell}+1)$ should be equal to 2. 

(iii) When $G$ is of type $E_6$, the highest root of $E_6$ is $\alpha_1 + 2\alpha_2 + 2\alpha_3 + 3\alpha_4 + 2\alpha_5 + \alpha_6$. 
Considering a simple root $\alpha_3$ of $E_6$, the Lie algebra $\mathfrak e_6$ has a gradation of depth 2 associated to $\alpha_3$. 
Because $\alpha_1 + \alpha_2 + 2\alpha_3 + 2\alpha_4 + \alpha_5$ is minimal among all roots in $\Phi_2$, 
we know that $t = 1 + \sum_{i \neq 3} \frac{c_i (a_i +1)}{2}>3$ for all $\alpha \in \Phi_2$. 
Since $\alpha_1+\alpha_3$ and $\alpha_3+\alpha_4$ are minimal roots in $\Phi_1$, 
either $1+(a_{1}+1)$ or $1+(a_{4}+1)$ should be equal to 2, which implies that one of either $a_{1}$ or $a_{4}$ must be 0. 
Next, for a gradation of $\mathfrak e_6$ of depth 3 associated to $\alpha_4$, 
$\alpha_2 + \alpha_3 + 2\alpha_4 + \alpha_5$ and $\alpha_1 + \alpha_2 + 2\alpha_3 + 3\alpha_4 + 2\alpha_5 + \alpha_6$ are minimal among all roots in $\Phi_2$ and $\Phi_3$, respectively. 
Thus, we get the inequalities $t = 1 + \sum_{i \neq 4} \frac{c_i (a_i +1)}{2}>2$ for all $\alpha \in \Phi_2$ and $t = 1 + \sum_{i \neq 4} \frac{c_i (a_i +1)}{3}>3$ for all $\alpha \in \Phi_3$.  

Likewise, we will check that the singular value 2 is attained at a minimal root in $\Phi_1$ for each cases in Sections 5, 7 and 8. 
\end{proof}

When $G/P_k$ is either a Hermitian symmetric space of compact type or an adjoint variety, 
the maximum of $\Sing(\omega)$ can be computed by the similar argument: 

\begin{proposition} \label{maximum of singular values} 
Assume that $G/P_k$ is a Hermitian symmetric space of compact type. 
If an irreducible equivariant vector bundle $\mathcal E_{\omega}$ on $G/P_k$ is Ulrich, 
then the maximum of $\Sing(\omega)$ is a singular value attained at the highest root $\theta$ of the Lie algebra $\mathfrak g$. 
\end{proposition}

\begin{proof}
An \emph{irreducible Hermitian symmetric space of compact type} is a compact homogeneous space $G/P$ with a simple Lie group $G$ and a maximal parabolic subgroup $P$ 
such that the isotropy representation of $P$ on the tangent space $T_o(G/P)$ at a base point $o$ is irreducible. 
Thus, if $G/P_k$ is a Hermitian symmetric space of compact type, 
then the decomposition of $\mathfrak g$ associated to a simple root $\alpha_k$ has a gradation of depth 1 and the parabolic subalgebra $\mathfrak p_k$ acts irreducibly on $\mathfrak g/\mathfrak p_k$. 
Equivalently, the coefficient $c_k$ of $\alpha_k$ in an expansion of the highest root of $\mathfrak g$ is 1. 
From Equation (\ref{singular values}) in the proof of Proposition \ref{coefficient of weight}, 
the singular value corresponding to a root $\alpha = \sum_{i=1}^n c_i \alpha_i$ in $\Phi^+ \setminus \Phi_0$ is given by $t = 1 + \sum_{i \neq k} \frac{c_i (\alpha_i, \alpha_i) (a_i +1)}{(\alpha_k, \alpha_k) }$ because $\Phi_1 = \Phi^+ \setminus \Phi_0$. 
Hence the maximum of $\Sing(\omega)$ is a singular value attained at the highest root $\theta$. 
\end{proof}

If $G/P_k$ is not a Hermitian symmetric space of compact type, then the decomposition of $\mathfrak g$ associated to a simple root $\alpha_k$ has a grading of depth greater than 1. 
In this case, we need to compare the singular values for maximal roots in each component $\mathfrak g_i$ of the graded Lie algebra in order to get the maximum of $\Sing(\omega)$. 

\begin{proposition} \label{maximum of singular values: adjoint variety}
Let $G/P_k$ be an adjoint variety such that $\mbox{\rm rank} (G) \geq 2$ and $G$ is not of type $A$. 
If an irreducible equivariant vector bundle $\mathcal E_{\omega}$ on $G/P_k$ is Ulrich, 
then the maximum of $\Sing(\omega)$ is a singular value attained at the root $\theta - \alpha_k$, where $\theta$ is the highest root of $\mathfrak g$. 
\end{proposition}

\begin{proof}
An \emph{adjoint variety} of $G$ is the closed orbit of a highest root vector in the adjoint representation of $G$. 
Since $G/P_k$ is an adjoint variety such that $\mbox{\rm rank} (G) \geq 2$ and $G$ is not of type $A$, 
by Theorem 4.1 of \cite{Ya}  
the simple Lie algebra $\mathfrak g$ admits a \emph{contact gradation}, that is, 
the decomposition $\mathfrak g = \mathfrak g_{-2} \oplus \mathfrak g_{-1} \oplus \mathfrak g_{0} \oplus \mathfrak g_{1} \oplus \mathfrak g_{2}$ associated to a simple root $\alpha_k$ has a gradation of depth 2, $\dim \mathfrak g_{2} = 1$, and the Lie bracket operation $[\, , \, ] \colon \mathfrak g_{1} \times \mathfrak g_{1} \to \mathfrak g_{2}$ is nondegenerate. 
Then we have $\Phi_2 = \{ \theta \}$ from $\dim \mathfrak g_{2} = 1$.
Moreover, since $[\, , \, ] \colon \mathfrak g_{1} \times \mathfrak g_{1} \to \mathfrak g_{2}$ is nondegenerate, 
for each $\alpha \in \Phi_1$ there exists $\beta \in \Phi_1$ such that $\alpha + \beta = \theta$.
Thus, $\Phi_1 = \{ \alpha \in \Phi^+ \mid \theta - \alpha \mbox{ is a root} \}$ and $\theta - \alpha_k$ is the maximal root in $\Phi_1$. 
The result follows from Equation (\ref{singular values}) in the proof of Proposition \ref{coefficient of weight}.
\end{proof}

\begin{remark}
For $G=\SL(n+1, \mathbb C)$, the adjoint representation on the Lie algebra $\mathfrak{sl}(n+1, \mathbb C)$  is an irreducible highest module $V_G(\omega_1 + \omega_n)$ 
and the adjoint variety of $G$ is a homogeneous space $G/(P_1 \cap P_n)$ with Picard number 2.
\end{remark}

\section{$G_2$-homogeneous varieties}

We follow notation of \cite{FH} for the basics on the representation theory of a Lie algebra. 
To begin with, we collect basic facts about the simple Lie group $G_2$. 
Let us choose an orthonormal basis $\{ L_1, L_2 \}$ for the dual Cartan subalgebra $\mathfrak t^*$. 
Then the simple roots can be taken as follows: 
$\alpha_1 = L_1$, $\alpha_2 = \frac{1}{2}(-3 L_1 + \sqrt{3} L_2)$.

\begin{center}
\begin{picture} (150, 35)
 \put(10,20){$(G_2)$}

 \put(80,20){\circle{5}} 
 \put(110,20){\circle{5}} 

 \put(82.5,18){\line(1,0){25}} 
 \put(82.5,20){\line(1,0){25}} 
 \put(82.5,22){\line(1,0){25.5}} 

 \put(93 ,17.5){$<$}

 \put(75,10){$\alpha_1$} 
 \put(106,10){$\alpha_2$} 
 \end{picture}
\end{center}
The complex Lie group $G_2$ has 6 positive roots:
$\Phi^+= \{ \alpha_1, \alpha_2, \alpha_1 + \alpha_2, 2 \alpha_1 + \alpha_2, 3\alpha_1 + \alpha_2, 3\alpha_1 + 2\alpha_2 \}
= \{ L_1, -\frac{3}{2} L_1 + \frac{\sqrt{3}}{2}L_2, -\frac{1}{2}L_1 + \frac{\sqrt{3}}{2}L_2, \frac{1}{2}L_1 + \frac{\sqrt{3}}{2}L_2, \frac{3}{2}L_1 + \frac{\sqrt{3}}{2}L_2,  \sqrt{3}L_2 \}$; 
see e.g. page 332 in \cite{FH}.
From the relation $(\omega_i, \alpha_j^{\vee})=\delta_{ij}$, 
the fundamental weights corresponding to the system of simple roots are 
$$\omega_1 = \frac{1}{2} (L_1 + \sqrt{3} L_2), \quad \omega_2 = \sqrt{3} L_2.$$
Thus we know that $\rho = \omega_1 + \omega_2 = \frac{1}{2} L_1 + \frac{3 \sqrt{3}}{2} L_2$.

\subsection{Homogeneous variety $G_2/P_1$}
It is well-known that the rational homogeneous variety $G_2/P_1$ is isomorphic to 
a 5-dimensional quadric $\mathbb{Q}^5$ in $\mathbb{P}^6$; see e.g. page 391 in \cite{FH}. 
The 
complex Lie group $G_2$ is the automorphism group of the Cayley algebra $\mathbb O$ of octonions with complex coefficients. 
In particular, the fundamental representation $V_{G_2}(\omega_1)$ of $G_2$ has dimension 7 
and is identified with the subspace $\mbox{Im } \mathbb O$ of imaginary octonions.
Since $G_2$ acts via algebra automorphisms, the action of $G_2$ on $\mbox{Im } \mathbb O$ preserves the multiplicative norm $q$.
From this fact, we have an inclusion $G_2 \subset \SO(7, \mathbb C)$. 
In fact, $G_2$ can be characterized as a stabilizer in $\GL(7, \mathbb C)$ 
of a nondegenerate quadratic form $q$ and a generic skew-symmetric 3-form $\Omega \in \wedge^3 (\mbox{Im } \mathbb O)^*$.

For a reductive part $L$ of the parabolic subgroup $P_1 \subset G_2$, 
an $L$-dominant integral weight $\omega$ is expressed as a linear combination 
$$\omega = a \omega_1 + b \omega_2 = \frac{1}{2} a L_1 + \Big(\frac{\sqrt{3}}{2} a+ \sqrt{3} b \Big) L_2$$
with $b \geq 0$.
Then $ \omega + \rho - t \, \omega_1 =  (\frac{1}{2}a + \frac{1}{2}-\frac{1}{2}t) L_1 + (\frac{\sqrt{3}}{2} a + \sqrt{3} b + \frac{3 \sqrt{3}}{2} - \frac{\sqrt{3}}{2}t) L_2 $.

\begin{proposition} \label{G2/P1}
There are no irreducible equivariant Ulrich bundles on 
$G_2/P_1$.
\end{proposition}

\begin{proof}
For $\omega \in \Lambda_L^+$, 
the weight $\omega+\rho-t\,\omega_1$
is singular if and only if it is orthogonal to one of roots of $G_2$. 
Considering the five positive roots 
$\alpha_1 = L_1, 
3\alpha_1 + \alpha_2 = \frac{3}{2}L_1 + \frac{\sqrt{3}}{2}L_2, 
2\alpha_1 +\alpha_2 = \frac{1}{2}L_1 + \frac{\sqrt{3}}{2}L_2, 
3\alpha_1 + 2\alpha_2 = \sqrt{3}L_2,
\alpha_1 +\alpha_2 = -\frac{1}{2}L_1 + \frac{\sqrt{3}}{2}L_2$ 
in $\Phi^+ \setminus \Phi_0$, 
we obtain the set of singular values 
$$\Sing(\omega)=\Big\{ a+1, a+1+(b+1), a+1+\frac{3}{2}(b+1), a+1+2(b+1), a+1+3(b+1) \Big\}. $$
For example, if the weight $\omega+\rho-t \, \omega_1$ is singular with respect to the positive root $\frac{3}{2}L_1 + \frac{\sqrt{3}}{2}L_2$,
then we have 
$\frac{3}{2}(\frac{1}{2}a + \frac{1}{2}-\frac{1}{2}t) + \frac{\sqrt{3}}{2}(\frac{\sqrt{3}}{2} a + \sqrt{3} b + \frac{3 \sqrt{3}}{2} - \frac{\sqrt{3}}{2}t) 
= \frac{3}{2}a + \frac{3}{2}b + 3 -\frac{3}{2}t = 0$,  
hence $t=a + b + 2$. 
By Proposition~\ref{coefficient of weight}, we have $a=0$ so that $\Sing(\omega)= \{ 1, b+2, \frac{3}{2}b+\frac{5}{2}, 2b+3, 3b+4 \}$.  
However, there is no integer $b$ such that $\Sing(\omega)=\{ 1,2,3,4,5 \}$. 
Therefore, there are no irreducible equivariant Ulrich bundles on the rational homogeneous variety $G_2/P_1$ by Lemma \ref{Ulrich criterion}.
\end{proof}

\begin{remark}
In Example \ref{Gr and Q5} (2), we show that the 5-dimensional hyperquadric 
$B_3/P_1 \cong \mathbb Q^5 \subset \mathbb{P}^6$ admits an Ulrich bundle.
Furthermore, we know that the spinor bundle of rank 4 is the only irreducible $B_3$-equivariant Ulrich bundle on $\mathbb Q^5$.
From the inclusion $G_2 \subset \SO(7, \mathbb C)$, 
we can regard this Ulrich bundle as a non-irreducible $G_2$-equivariant bundle on $G_2/P_1 \cong \mathbb Q^5$. 
\end{remark}

\subsection{Homogeneous variety $G_2/P_2$}
The rational homogeneous variety $G_2/P_2 \subset \mathbb P(V_{G_2}(\omega_2))=\mathbb P^{13}$ has dimension 5 and Fano index 3
(see Section 9 of \cite{Snow} for dimensions and Fano indexes of rational homogeneous varieties). 
Note that $G_2/P_2$ is the adjoint variety of $G_2$ 
since the second fundamental representation $V_{G_2}(\omega_2)$ is the adjoint representation $\mathfrak g_2$.

\begin{proposition}
There are no irreducible equivariant Ulrich bundles on 
$G_2/P_2$.
\end{proposition}

\begin{proof}
By computations for $\omega + \rho - t \, \omega_2 = (\frac{1}{2}a + \frac{1}{2}) L_1 + (\frac{\sqrt{3}}{2} a + \sqrt{3} b + \frac{3 \sqrt{3}}{2} - \sqrt{3} t) L_2 $ as in the proof of Proposition~\ref{G2/P1}, 
we get the set $\Sing(\omega)$ of singular values. 
Considering the five positive roots 
$-\frac{3}{2} L_1 + \frac{\sqrt{3}}{2}L_2, 
-\frac{1}{2}L_1 + \frac{\sqrt{3}}{2}L_2, 
\sqrt{3}L_2,  
\frac{1}{2}L_1 + \frac{\sqrt{3}}{2}L_2, 
\frac{3}{2}L_1 + \frac{\sqrt{3}}{2}L_2$ 
in $\Phi^+ \setminus \Phi_0$, 
we obtain that 
$$\Sing(\omega)=\Big\{ b+1, \frac{1}{3}(a + 1) + (b+1), \frac{1}{2}(a+1) + (b+1), \frac{2}{3}(a+1) + (b+1), (a+1) + (b+1) \Big\}. $$
By Proposition \ref{coefficient of weight}, we have $b=0$ so that $\Sing(\omega)= \{ 1, \frac{1}{3}a+\frac{4}{3}, \frac{1}{2}a+\frac{3}{2}, \frac{2}{3}a+\frac{5}{3}, a+2 \}$.  
However, there is no integer $a$ such that $ \Sing(\omega)=\{ 1,2,3,4,5 \} $. 
Therefore, there are no irreducible equivariant Ulrich bundles on the rational homogeneous variety $G_2/P_2$ by Lemma \ref{Ulrich criterion}.
\end{proof}

\section{$E_6$-homogeneous varieties}

We recall some facts about the simple Lie group $E_6$. 
When we choose an orthonormal basis $\{L_1, \cdots, L_6\}$ for the dual Cartan subalgebra $\mathfrak t^*$, 
the simple roots can be taken as follows: 
$\alpha_1 = \frac{1}{2} (L_1-L_2-L_3-L_4-L_5+\sqrt{3}L_6)$,  
$\alpha_2 = L_1 + L_2$, $\alpha_3 = L_2 - L_1$, $\alpha_4 = L_3 - L_2$, $\alpha_5 = L_4 - L_3$, $\alpha_6 = L_5 - L_4$.

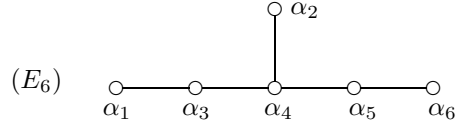
\begin{center}
\begin{picture} (180, 55)
 \put(10,20){$(E_6)$}

 \put(50,20){\circle{5}} 
 \put(110,50){\circle{5}} 
 \put(80,20){\circle{5}} 
 \put(110,20){\circle{5}} 
 \put(140,20){\circle{5}} 
 \put(170,20){\circle{5}} 

 \put(52.5,20){\line(1,0){25}} 
 \put(82.5,20){\line(1,0){25}} 
 \put(112.5,20){\line(1,0){25}} 
 \put(142.5,20){\line(1,0){25}} 
 \put(110,22.5){\line(0,1){25}} 

 \put(45,10){$\alpha_1$} 
 \put(116,48){$\alpha_2$} 
 \put(75,10){$\alpha_3$} 
 \put(106,10){$\alpha_4$} 
 \put(138,10){$\alpha_5$} 
 \put(168,10){$\alpha_6$} 
 \end{picture}
\end{center}
The Lie group $E_6$ has 36 positive roots; 
see e.g. page 333 in \cite{FH}:
$$\Phi^+=\{ L_i+L_j\}_{1\leq i<j\leq5} \cup \{ L_i-L_j\}_{1\leq j<i\leq5} \cup \Big \{ \frac{1}{2} (\sum_{i=1}^{5} (-1)^{n(i)}L_i + \sqrt{3} L_6) : \sum_{i=1}^{5} n(i) \mbox{ is even}  \Big\}.$$
The highest root of $E_6$ is $\alpha_1 + 2\alpha_2 + 2\alpha_3 + 3\alpha_4 + 2\alpha_5 + \alpha_6$.

The fundamental weights corresponding to the system of simple roots are 
\begin{eqnarray*}
\omega_1 &=& \frac{2\sqrt{3}}{3}L_6, \\
\omega_2 &=& \frac{1}{2}(L_1+L_2+L_3+L_4+L_5) + \frac{\sqrt{3}}{2}L_6, \\
\omega_3 &=& \frac{1}{2}(-L_1+L_2+L_3+L_4+L_5) + \frac{5\sqrt{3}}{6}L_6, \\
\omega_4 &=& L_3+L_4+L_5 + \sqrt{3}L_6, \\
\omega_5 &=& L_4+L_5 + \frac{2\sqrt{3}}{3}L_6, \\
\omega_6 &=& L_5 + \frac{\sqrt{3}}{3}L_6.
\end{eqnarray*}
Thus we know that $\rho = \sum_{i=1}^{6} \omega_i = L_2 + 2L_3 + 3L_4 + 4L_5 + 4\sqrt{3}L_6$.

\subsection{Homogeneous variety $E_6/P_1 \cong E_6/P_6$}

The \emph{(complex) Cayley plane} $E_6/P_1$ is the closed orbit 
in the projectivization of the minimal representation $V_{E_6}(\omega_1)$ of the simply-connected complex Lie group $E_6$. 
This is a smooth projective variety of dimension 16 and Fano index 12. 
The Cayley plane is one of the two exceptional Hermitian symmetric spaces of compact type, 
and appears as the last Severi variety.  
This can also be identified with the variety $\mathbb{OP}^2 \subset \mathbb P(\mathcal J_3(\mathbb O))$ of rank 1 matrices 
in the exceptional simple complex Jordan algebra $\mathcal J_3(\mathbb O)$ consisting of octonionic-Hermitian matrices of order 3.
Because the highest weight $E_6$-module $V_{E_6}(\omega_6)$ is dual to $V_{E_6}(\omega_1) = \mathcal J_3(\mathbb O) \cong \mathbb C^{27}$, 
the dual Cayley plane $E_6/P_6 \subset \mathbb P^{26}$ is projectively equivalent to $E_6/P_1\subset \mathbb P^{26}$.

For a reductive part $L$ of $P_1 \subset E_6$, an $L$-dominant weight $\omega$ is expressed as a linear combination 
\begin{eqnarray*}
\omega &=& a\omega_1 + b\omega_2 + c\omega_3 + d\omega_4 + e\omega_5 + f\omega_6 \\
&=& \frac{1}{2}(b-c)L_1 +  \frac{1}{2}(b+c)L_2 + \Big (\frac{1}{2}b+\frac{1}{2}c+d \Big)L_3  + \Big(\frac{1}{2}b+\frac{1}{2}c+d+e \Big)L_4 + \Big(\frac{1}{2}b+\frac{1}{2}c+d+e+f \Big)L_5 \\
&& + \Big(\frac{2\sqrt{3}}{3}a+\frac{\sqrt{3}}{2}b+\frac{5\sqrt{3}}{6}c+\sqrt{3}d+\frac{2\sqrt{3}}{3}e+\frac{\sqrt{3}}{3}f \Big)L_6 
\qquad \text{with $b,c,d,e,f \geq 0$.}
\end{eqnarray*}

\begin{proposition}\label{E_6/P_1}
The only irreducible equivariant Ulrich bundle on the Cayley plane $E_6/P_1$ is $\mathcal E_{\omega_5+3\omega_6}$.
\end{proposition}

\begin{proof}
For $\omega \in \Lambda_L^+$, 
the weight $\omega+\rho-t\,\omega_1= \frac{1}{2}(b-c)L_1 +  \{\frac{1}{2}(b+c)+1\}L_2 + (\frac{1}{2}b+\frac{1}{2}c+d+2)L_3 + (\frac{1}{2}b+\frac{1}{2}c+d+e+3)L_4 + 
(\frac{1}{2}b+\frac{1}{2}c+d+e+f+4)L_5 + \{ \frac{2\sqrt{3}}{3}(a-t)+\frac{\sqrt{3}}{2}b+\frac{5\sqrt{3}}{6}c+\sqrt{3}d+\frac{2\sqrt{3}}{3}e+\frac{\sqrt{3}}{3}f+4\sqrt{3}\}L_6$
is singular if and only if it is orthogonal to one of roots of $E_6$. 
To compute the set $\Sing(\omega)$, it suffices to consider the 16 positive roots 
$\{ \frac{1}{2} (\sum_{i=1}^{5} (-1)^{n(i)}L_i + \sqrt{3} L_6) : \sum_{i=1}^{5} n(i) \mbox{ is even}\}$ in $\Phi_1 = \Phi^+ \setminus \Phi_0$.

Suppose that an irreducible equivariant vector bundle $\mathcal E_{\omega}$ on the Cayley plane $E_6/P_1$ is Ulrich. 
Then $a=0$ by Proposition \ref{coefficient of weight} and 
$c=0$ by Proposition \ref{subminimum} 
because the Lie algebra $\mathfrak e_6$ has a gradation of depth 1 associated to $\alpha_1$. 
By straightforward computations, we get the set of singular values 




$ \Sing(\omega)=\{
1,
2,
d + 3,
b + d + 4,
d + e + 4,
b + d + e + 5,
d + e + f + 5, 
b + d + e + f + 6,
b + 2d + e + 6,
b + 2d + e + f + 7,
b + 2d + e + 7,
b + 2d + e + f + 8,
b + 2d + 2e + f + 8,
b + 2d + 2e + f + 9,
b + 3d + 2e + f + 10,
2b + 3d + 2e + f + 11
\}. $ 

Because $\Sing(\omega) = \{1, 2, \cdots, 16\}$ by Lemma \ref{Ulrich criterion}, 
we deduce that $d + 3 = 3$, $b + 3d + 2e + f + 10=15$ and $2b + 3d + 2e + f + 11=16$, which imply $d=0$, $b=0$ and $2e + f =5$. 
From $e+4=5$, we get $e=1$, $f=3$ and we can check that $\Sing(\omega_5+3\omega_6)=\{ 1, 2, \cdots, 16\}.$ 
Hence the only irreducible equivariant Ulrich bundle on the Cayley plane $E_6/P_1$ is $\mathcal E_{\omega_5+3\omega_6}$.
\end{proof}

\begin{remark} 
The dimension of the irreducible $\Spin(10)$-module with highest weight $3\varpi_1+\varpi_2$ is $4608=2^9 \times 3^2$ by Weyl's character formula,  
hence the unique irreducible equivariant Ulrich bundle $\mathcal E_{\omega_5+3\omega_6}$ on the Cayley plane $E_6/P_1$ has rank 4,608.
\end{remark}

\subsection{Homogeneous variety $E_6/P_2$}
The rational homogeneous variety $E_6/P_2 \subset \mathbb P(V_{E_6}(\omega_2))=\mathbb P^{77}$ has dimension 21 and Fano index 11. 
Note that $E_6/P_2$ is the adjoint variety of $E_6$ 
since the fundamental $E_6$-module $V_{E_6}(\omega_2)$ is the adjoint representation $\mathfrak e_6$.

\begin{proposition}\label{E_6/P_2}
There are no irreducible equivariant Ulrich bundles on $E_6/P_2$.
\end{proposition}

\begin{proof} 
Let $\mathcal{E}_{\omega}$ be an irreducible equivariant bundle on $E_6/P_2$. 
If $\mathcal{E}_{\omega}$ is Ulrich, then $b=d=0$ by Proposition \ref{coefficient of weight} and Proposition \ref{subminimum} 
since the decomposition of $\mathfrak e_6$ associated to $\alpha_2$ is a contact gradation. 
From a similar computation as above, we can compute the set $\Sing(\omega)$ as follows: 


$ \Sing(\omega)=\{
1,
2,
c + 3,
e + 3,
a + c + 4, 
c + e + 4,
e + f + 4,
c + e + 5,
c + e + f + 5,
a + c + e + 5,
\frac{1}{2}a + c + e + \frac{1}{2}f + \frac{11}{2},
c + e + f + 6,
a + c + e + 6,
a + c + e + f + 6,
c + 2e + f + 7,
a + 2c + e + 7,
a + c + e + f + 7,
a + c + 2e + f + 8,
a + 2c + e + f + 8,
a + 2c + 2e + f + 9,
a + 2c + 2e + f + 10
\} $.

Suppose that there are $a,c,e,f$ such that $\Sing(\omega)=\{ 1, 2, \cdots, 21 \}.$ 
Then 
either $c=0, e \neq 0$ or $c \neq 0, e=0$. 
As $c + e + f + 6 \neq a + c + e + f + 6$, 
we see that $a \neq 0$. 
Because $a + 2c + 2e + f + 10$ 
is the largest integer in $\Sing(\omega)$, we have $a + 2c + 2e + f + 10=21$ so that $a+f$ is odd. 

(i) If $c=0, e \neq 0$, then we have $e=1$, $a+f=9$. 
Then we have either $a=3,f=6$ or $a=7, f=2$. 
However, we can check that the both cases cannot give $\Sing(\omega)=\{ 1, 2, \cdots,21 \}$.

(ii) If $c \neq 0, e=0$, then we have $c=1, f \neq 0$ or $c=2, f=0$. 
If $c=1, f \neq 0$, then we have either $a=2, f=7$ or $a=6, f=3$. 
However, we can check that no case can give $\Sing(\omega)=\{ 1,\cdots,21 \}$ by direct computation.
Likewise, the case $c=2, f=0$ cannot occur since $c + e + 5= c + e + f + 5$. 

Therefore, we see that there are no $a,b,c,d,e,f$ such that $\Sing(\omega)=\{ 1, 2, \cdots,21 \}$, 
so that there are no irreducible equivariant Ulrich bundles on the rational homogeneous variety $E_6/P_2$ by Lemma \ref{Ulrich criterion}. 
\end{proof}

\subsection{Homogeneous variety $E_6/P_3 \cong E_6/P_5$}
The rational homogeneous variety $E_6/P_3$ has dimension 25 and Fano index 9. 
Note that the rational homogeneous variety $E_6/P_5 \subset \mathbb P^{350}$ is projectively equivalent to $E_6/P_3 \subset \mathbb P^{350}$
because the highest weight $E_6$-module $V_{E_6}(\omega_5)$ is dual to $V_{E_6}(\omega_3)$. 

\begin{proposition}\label{E_6/P_3}
There are no irreducible equivariant Ulrich bundles on $E_6/P_3$.
\end{proposition}

\begin{proof}
Let $\mathcal{E}_{\omega}$ be an irreducible equivariant bundle on $E_6/P_3$. 
If $\mathcal{E}_{\omega}$ is Ulrich, then $c=0$ by Proposition \ref{coefficient of weight} and either $a=0$ or $d=0$ by Proposition \ref{subminimum}. 
Considering the five positive roots in $\Phi_2=\{ 
\alpha_1 + \alpha_2 + 2\alpha_3 + 2\alpha_4 + \alpha_5, 
\alpha_1 + \alpha_2 + 2\alpha_3 + 2\alpha_4 + \alpha_5 + \alpha_6, 
\alpha_1 + \alpha_2 + 2\alpha_3 + 2\alpha_4 + 2\alpha_5 + \alpha_6, 
\alpha_1 + \alpha_2 + 2\alpha_3 + 3\alpha_4 + 2\alpha_5 + \alpha_6, 
\alpha_1 + 2\alpha_2 + 2\alpha_3 + 3\alpha_4 + 2\alpha_5 + \alpha_6\}$, 
we check that 
$\{ 
\frac{1}{2}a + \frac{1}{2}b + d + \frac{1}{2}e + \frac{7}{2}, 
\frac{1}{2}a + \frac{1}{2}b + d + \frac{1}{2}e + \frac{1}{2}f + 4,
\frac{1}{2}a + \frac{1}{2}b + d + e + \frac{1}{2}f + \frac{9}{2},
\frac{1}{2}a + \frac{1}{2}b + \frac{3}{2}d + e + \frac{1}{2}f + 5,
\frac{1}{2}a + b + c + \frac{3}{2}d + e + \frac{1}{2}f + \frac{11}{2}
\} \subset \Sing(\omega)$.
%
%
%
On the other hand, since  
$
\frac{1}{2}a + \frac{1}{2}b + d + e + \frac{1}{2}f + \frac{9}{2},
\frac{1}{2}a + \frac{1}{2}b + \frac{3}{2}d + e + \frac{1}{2}f + 5$
are integers, we see that $d$ 
is odd. 
Hence $a=0$. 
Similarly, $f$ is an odd integer since 
$\frac{1}{2}a + \frac{1}{2}b + d + \frac{1}{2}e + \frac{1}{2}f + 4,
\frac{1}{2}a + \frac{1}{2}b + d + \frac{1}{2}e + \frac{7}{2}$ are integers. 
However, $\frac{1}{2}a + b + c + \frac{3}{2}d + e + \frac{1}{2}f + \frac{11}{2}$ cannot be an integer.
Therefore, 
there are no $a,b,c,d,e,f$ such that $\Sing(\omega)=\{ 1, 2, \cdots,25 \}.$
\end{proof}

\subsection{Homogeneous variety $E_6/P_4$}
The rational homogeneous variety $E_6/P_4 \subset \mathbb P^{2924}$ has dimension 29 and Fano index 7. 

\begin{proposition}
There are no irreducible equivariant Ulrich bundles on 
$E_6/P_4$.
\end{proposition}

\begin{proof}
Let $\mathcal{E}_{\omega}$ be an irreducible equivariant bundle on $E_6/P_4$. 
If $\mathcal{E}_{\omega}$ is Ulrich, then $d=0$ by Proposition \ref{coefficient of weight} and at least one of $b, c, e$ is zero 
by Proposition \ref{subminimum}. 
Considering the two positive roots in $\Phi_3=\{ 
\alpha_1 + \alpha_2 + 2\alpha_3 + 3\alpha_4 + 2\alpha_5 + \alpha_6, 
\alpha_1 + 2\alpha_2 + 2\alpha_3 + 3\alpha_4 + 2\alpha_5 + \alpha_6\}$, 
we check that 
$\{ 
\frac{1}{3}a + \frac{1}{3}b + \frac{2}{3}c + \frac{2}{3}e + \frac{1}{3}f + \frac{10}{3},
\frac{1}{3}a + \frac{2}{3}b + \frac{2}{3}c + \frac{2}{3}e + \frac{1}{3}f + \frac{11}{3}
\} \subset \Sing(\omega)$.
Since these are integers, $b$ is an integer such that $b \equiv 2 \mod 3$. 
Similarly, considering the nine positive roots in $\Phi_2$, 
we check that 
$\{ 
\frac{1}{2}b + \frac{1}{2}c + \frac{1}{2}e + \frac{1}{2}f + 3,
\frac{1}{2}b + \frac{1}{2}c + e + \frac{1}{2}f + \frac{7}{2}, 
\frac{1}{2}a + \frac{1}{2}b + \frac{1}{2}c + d + e + \frac{1}{2}f + 4, 
\frac{1}{2}a + \frac{1}{2}b + c + d + e + \frac{1}{2}f + \frac{9}{2}
\} \subset \Sing(\omega)$.
Because
$\frac{1}{2}b + \frac{1}{2}c + \frac{1}{2}e + \frac{1}{2}f + 3$ and 
$\frac{1}{2}b + \frac{1}{2}c + e + \frac{1}{2}f + \frac{7}{2}$ are integers, 
we see that $e$ is an odd integer. 
Thus we conclude that $c=0$. 
However, if $c=0$ then 
$\frac{1}{2}a + \frac{1}{2}b + \frac{1}{2}c + e + \frac{1}{2}f + 4$ and $\frac{1}{2}a + \frac{1}{2}b + c + e + \frac{1}{2}f + \frac{9}{2}$ cannot be integers simultaneously. 
Therefore,
there are no $a,b,c,d,e,f$ such that $\Sing(\omega)=\{ 1, 2, \cdots, 29 \}.$
\end{proof}

\section{$F_4$-homogeneous varieties}

We recall some facts about the simple Lie group $F_4$. 
When we choose an orthonormal basis $\{L_1, L_2, L_3, L_4\}$ for the dual Cartan subalgebra $\mathfrak t^*$, 
the simple roots can be taken as follows: 
$\alpha_1 = L_2 - L_3$, $\alpha_2 = L_3 - L_4$, $\alpha_3 = L_4$, $\alpha_4 = \frac{1}{2} (L_1-L_2-L_3-L_4)$.
The simple Lie group $F_4$ can be identified with a subgroup of $E_6$. 
A symmetry of the Dynkin diagram of $E_6$ induces an automorphism $\sigma \colon E_6 \to E_6$ of order 2. 
Then we know that the subgroup $E_6^{\sigma}$ of $\sigma$-stable elements of $E_6$ is isomorphic to $F_4$; 
see Proposition 13.31 of \cite{Carter}.

\begin{center}
\begin{picture} (150, 40)
 \put(10,20){$(F_4)$}

 \put(50,20){\circle{5}} 
 \put(80,20){\circle{5}} 
 \put(110,20){\circle{5}} 
 \put(140,20){\circle{5}} 

 \put(52.5,20){\line(1,0){25}} 
 \put(82.5,19){\line(1,0){25}} 
 \put(82.5,21){\line(1,0){25}} 
 \put(112.5,20){\line(1,0){25.5}} 

 \put(93 ,17.5){$>$}

 \put(45,10){$\alpha_1$} 
 \put(75,10){$\alpha_2$} 
 \put(106,10){$\alpha_3$} 
 \put(138,10){$\alpha_4$} 
 \end{picture}
\end{center}
The complex Lie group $F_4$ has 24 positive roots; 
see e.g. page 332 in \cite{FH}:
$$\Phi^+=\{ L_i\}_{1\leq i \leq4} \cup \{ L_i+L_j\}_{1\leq i<j\leq4} \cup \{ L_i-L_j\}_{1\leq i<j\leq4} \cup \Big \{ \frac{1}{2} (L_1 \pm L_2 \pm L_3 \pm L_4) \Big\}.$$ 
The highest root of $F_4$ is $2\alpha_1 + 3\alpha_2 + 4\alpha_3 + 2\alpha_4$.

The fundamental weights corresponding to the system of simple roots are 
$$\omega_1 = L_1 + L_2, \quad \omega_2 = 2L_1 + L_2 + L_3, \quad \omega_3 = \frac{1}{2} (3L_1+L_2+L_3+L_4), \quad \omega_4 = L_1.$$
Thus we know that $\rho = \omega_1 + \omega_2 + \omega_3 + \omega_4 = \frac{11}{2}L_1 + \frac{5}{2}L_2 + \frac{3}{2}L_3 + \frac{1}{2}L_4$.

\subsection{Homogeneous variety $F_4/P_4$}
For the maximal parabolic subgroup $P_4$ associated to a simple root $\alpha_4$ of $F_4$, 
the rational homogeneous variety $F_4/P_4$ is a general hyperplane section of the Cayley plane $E_6/P_1 \subset \mathbb P^{26}$; 
see Section 6.3 of \cite{LM}.
This geometric description of $F_4/P_4$ follows from the fact that 
the 27-dimensional fundamenatal $E_6$-module $V_{E_6}(\omega_1)$ may be regarded as an $F_4$-module using the embedding of $F_4$ in $E_6$
and it decomposes as $V_{F_4}(\omega_4) \oplus V_{F_4}(0)$ (Proposition 13.32 of \cite{Carter}). 
Note that $V_{F_4}(\omega_4)$ can be identified with the taceless subspace of the Jordan algebra $\mathcal J_3(\mathbb O)$ explained in Subsection 5.1. 
So $F_4/P_4 \subset \mathbb P^{25}$ is a smooth projective variety of dimension 15 and Fano index 11. 


\begin{proposition}
There are no irreducible equivariant Ulrich bundles on 
$F_4/P_4$.
\end{proposition}

\begin{proof}
For $\omega \in \Lambda_L^+$, 
suppose that an irreducible equivariant vector bundle $\mathcal E_{\omega}$ on 
$F_4/P_4$ is Ulrich. 
Then $\Sing(\omega) = \{1, 2, \cdots, 15\}$ by Lemma \ref{Ulrich criterion} and the fourth coefficient $d$ of $\omega$ must be zero by Proposition \ref{coefficient of weight}. 
Considering the positive roots in $ \Phi^+ \setminus \Phi_0$, 
we check that $\{ a + 2b + \frac{3}{2}c + \frac{11}{2}, a + 2b + 2c + 6, 2a + 2b + c + 6, 2a + 4b + 3c + 10 \} \subset \Sing(\omega)$.
As $a + 2b + \frac{3}{2}c + \frac{11}{2}$ is an integer, $c$ is an odd integer. 
Since $2a + 4b + 3c + 10 \leq 15$, $c=1$, $b=0$, and $a=0$ or $1$.  
If $a=0$, then $a + 2b + \frac{3}{2}c + \frac{11}{2} = 7 = 2a + 2b + c + 6$. 
Similarly, if $a=1$, then $a + 2b + 2c + 6 = 9 = 2a + 2b + c + 6$.  
Hence there are no integers $a,b,c,d$ such that $\Sing(\omega)=\{ 1,2, \cdots, 15 \}$.  
\end{proof}

Although the rational homogeneous variety $F_4/P_4$ does not admit an irreducible equivariant Ulrich bundle, 
we can constuct a (non-irreducible) equivariant Ulrich bundle on $F_4/P_4$ 
because $F_4/P_4$ is a general hyperplane section of the Cayley plane $E_6/P_1$. 
This follows from a fact in \cite{ES} and \cite{Be17}: 
if $E$ is an Ulrich bundle on $X$ and $Y$ is a hyperplane section of $X$, then the restriction $E|_{Y}$ to $Y$ is also an Ulrich bundle on $Y$. 

\begin{corollary}
Let $E$ be the irreducible equivariant Ulrich bundle 
on the Cayley plane $E_6/P_1$ in Proposition \ref{E_6/P_1}.
Then the restriction $E|_{F_4/P_4}$ is an ($F_4$-equivariant) Ulrich bundle on 
$F_4/P_4$.
\end{corollary}

\subsection{Homogeneous variety $F_4/P_3$}
For the maximal parabolic subgroup $P_3$ associated to a simple root $\alpha_3$ of $F_4$, 
the rational homogeneous variety $F_4/P_3 \subset \mathbb P^{272}$ is 
the closed $F_4$-orbit in the space of lines on the rational homogeneous variety $F_4/P_4 \subset \mathbb P^{25}$
(cf. Section 6.4 of 
\cite{LM}), 
and has dimension 20 and Fano index 7.








\begin{proposition}
There are no irreducible equivariant Ulrich bundles on 
$F_4/P_3$.
\end{proposition}

\begin{proof}
Suppose that an irreducible equivariant vector bundle $\mathcal E_{\omega}$ on $F_4/P_3$ is Ulrich. 
Then $\Sing(\omega) = \{1, 2, \cdots, 20\}$ by Lemma \ref{Ulrich criterion} and the third coefficient $c$ of $\omega$ is zero by Proposition \ref{coefficient of weight}. 
Considering the positive roots in $ \Phi^+ \setminus \Phi_0$, 
we check that $\{ \frac{1}{2}a + b + \frac{1}{2} d + 3, \frac{1}{2}a + \frac{3}{2}b + \frac{1}{2} d + \frac{7}{2} \} \subset \Sing(\omega)$ by  straightforward computations.
Since the difference $b+ \frac{1}{2}$ between two elements cannot be an integer, there are no integers $a,b,c,d$ such that $\Sing(\omega)=\{ 1,2, \cdots, 20 \}$.
%
\end{proof}
%

\subsection{Homogeneous variety $F_4/P_2$}
The rational homogeneous variety $F_4/P_2 \subset \mathbb P^{1273}$ has dimension 20 and Fano index 5. 

\begin{proposition}
There are no irreducible equivariant Ulrich bundles on 
$F_4/P_2$.
\end{proposition}

\begin{proof}
Suppose that an irreducible equivariant vector bundle $\mathcal E_{\omega}$ on $F_4/P_2$ is Ulrich. 
Then $\Sing(\omega) = \{1, 2, \cdots, 20\}$ by Lemma \ref{Ulrich criterion} and the second coefficient $b$ of $\omega$ is zero by Proposition \ref{coefficient of weight}. 
By straightforward computations, 
we check that $\{ \frac{2}{3}a + \frac{2}{3}c + \frac{1}{3}d + \frac{8}{3}, \frac{1}{3}a + \frac{2}{3}c + \frac{1}{3}d + \frac{7}{3}, 
\frac{1}{2}a + \frac{3}{4}c + \frac{1}{2}d + \frac{11}{4}, \frac{1}{2}a + \frac{3}{4}c + \frac{1}{4}d + \frac{5}{2}, 
\frac{1}{2}a + \frac{1}{2}c + \frac{1}{2}d + \frac{5}{2}, a + c + d + 4 \} \subset \Sing(\omega)$.
Because $\frac{2}{3}a + \frac{2}{3}c + \frac{1}{3}d + \frac{8}{3}$ and $\frac{1}{3}a + \frac{2}{3}c + \frac{1}{3}d + \frac{7}{3}$ are integers, 
we see that $a$ is an integer such that $a \equiv 2 \mod 3$. 
Similarly, $c$ and $d$ are integers such that $c, d \equiv 3 \mod 4$. 
Then $a$ must be odd since $\frac{1}{2}a + \frac{1}{2}c + \frac{1}{2}d + \frac{5}{2}$ is an integer. 
Thus $a=5$ from $a + c + d + 4 \leq 20$.
If $a=5$, then we get $c=d=3$ since $\frac{2}{3}c + \frac{1}{3}d $ is an integer. 
However, we conclude that $\frac{1}{2}a + \frac{3}{4}c + \frac{1}{2}d + \frac{11}{4} = 9 = \frac{2}{3}a + \frac{2}{3}c + \frac{1}{3}d + \frac{8}{3}$. 
Therefore, there are no integers $a,b,c,d$ such that $\Sing(\omega)=\{ 1, 2, \cdots, 20 \}$.
\end{proof}

\subsection{Homogeneous variety $F_4/P_1$}
The rational homogeneous variety $F_4/P_1 \subset \mathbb P^{51}$ has dimension 15 and Fano index 8. 
Since the fundamental $F_4$-module $V_{F_4}(\omega_1)$ is the adjoint representation $\mathfrak f_4$, $F_4/P_1$ is the adjoint variety of $F_4$.

\begin{proposition}
There are no irreducible equivariant Ulrich bundles on 
$F_4/P_1$.
\end{proposition}

\begin{proof}
Let $\mathcal{E}_{\omega}$ be an irreducible equivariant bundle on $F_4/P_1$. 
If $\mathcal{E}_{\omega}$ is Ulrich, then $a=0$ by Proposition \ref{coefficient of weight} and $b=0$ by Proposition \ref{subminimum}
since the singular value $2$ is attained at the root $\alpha_1 + \alpha_2 \in \Phi_1$, that is, $(a+1)+(b+1) = 2$.
By straightforward computations, 
we check that 
$\{ 
a + b + c + \frac{1}{2}d + \frac{7}{2},
a + \frac{3}{2}b + c + \frac{1}{2}d + 4
\} 
\subset \Sing(\omega)$.
%
%
%
%
As $a + b + c + \frac{1}{2}d + \frac{7}{2}$ is an integer, $d$ is an odd integer.
However, if $d$ is odd, then $a + \frac{3}{2}b + c + \frac{1}{2}d + 4$ cannot be an integer.
Therefore, 
there are no $a,b,c,d$ such that $\Sing(\omega)=\{ 1, 2, \cdots, 15 \}.$
\end{proof}

\section{$E_7$-homogeneous varieties}
We recall some basic facts about the simple Lie group $E_7$. 
When we choose an orthonormal basis $\{L_1, \cdots, L_7\}$ for the dual Cartan subalgebra $\mathfrak t^*$, 
the simple roots can be taken as follows: \\
$\alpha_1 = \frac{1}{2} (L_1-L_2-L_3-L_4-L_5-L_6+\sqrt{2}L_7)$,  
$\alpha_2 = L_1 + L_2$, $\alpha_3 = L_2 - L_1$, $\alpha_4 = L_3 - L_2$, 
$\alpha_5 = L_4 - L_3$, $\alpha_6 = L_5 - L_4$, $\alpha_7 = L_6 - L_5$.

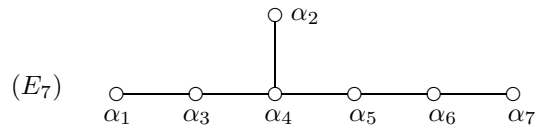
\begin{center}
\begin{picture} (250, 55)
 \put(10,20){$(E_7)$}

 \put(50,20){\circle{5}} 
 \put(110,50){\circle{5}} 
 \put(80,20){\circle{5}} 
 \put(110,20){\circle{5}} 
 \put(140,20){\circle{5}} 
 \put(170,20){\circle{5}} 
 \put(200,20){\circle{5}} 

 \put(52.5,20){\line(1,0){25}} 
 \put(82.5,20){\line(1,0){25}} 
 \put(112.5,20){\line(1,0){25}} 
 \put(142.5,20){\line(1,0){25}} 
 \put(110,22.5){\line(0,1){25}} 
 \put(172.5,20){\line(1,0){25}} 

 \put(45,10){$\alpha_1$} 
 \put(116,48){$\alpha_2$} 
 \put(75,10){$\alpha_3$} 
 \put(106,10){$\alpha_4$} 
 \put(138,10){$\alpha_5$} 
 \put(168,10){$\alpha_6$} 
 \put(198,10){$\alpha_7$} 
 \end{picture}
\end{center}
The complex Lie group $E_7$ has 63 positive roots, see e.g. page 333 in \cite{FH}:
$$\Phi^+=\{ L_i+L_j\}_{1 \leq i < j \leq 6} \cup \{ L_i-L_j\}_{1 \leq j < i \leq 6} \cup \{ \sqrt{2} L_7 \}  
\cup \Big \{ \frac{1}{2} (\sum_{i=1}^{6} (-1)^{n(i)}L_i + \sqrt{2} L_7) : \sum_{i=1}^{6} n(i) \mbox{ is odd} \Big \}.$$
The highest root of $E_7$ is $2\alpha_1 + 2\alpha_2 + 3\alpha_3 + 4\alpha_4 + 3\alpha_5 + 2\alpha_6 + \alpha_7$.

The fundamental weights corresponding to the system of simple roots are 
\begin{eqnarray*}
\omega_1 &=& \sqrt{2}L_7, \\
\omega_2 &=& \frac{1}{2}(L_1+L_2+L_3+L_4+L_5+L_6) + \sqrt{2}L_7, \\
\omega_3 &=& \frac{1}{2}(-L_1+L_2+L_3+L_4+L_5+L_6) + \frac{3\sqrt{2}}{2}L_7, \\
\omega_4 &=& L_3+L_4+L_5+L_6 + 2\sqrt{2}L_7, \\
\omega_5 &=& L_4+L_5+L_6 + \frac{3\sqrt{2}}{2}L_7, \\
\omega_6 &=& L_5+L_6 + \sqrt{2}L_7, \\
\omega_7 &=& L_6 + \frac{\sqrt{2}}{2}L_7.
\end{eqnarray*}
Thus we know that $\rho= L_2 + 2L_3 + 3L_4 + 4L_5 + 5L_6+ \frac{17\sqrt{2}}{2}L_7$. 


\subsection{Homogeneous variety $E_7/P_1$}
Since the fundamental $E_7$-module $V_{E_7}(\omega_1)$ is the adjoint representation $\mathfrak e_7$,  
the rational homogeneous variety $E_7/P_1$ is the adjoint variety of $E_7$. 
The $E_7$-adjoint variety $E_7/P_1 \subset \mathbb P^{132}$ has dimension 33 and Fano index 17. 

\begin{proposition}\label{E_7/P_1}
There is a unique irreducible equivariant Ulrich bundle $\mathcal{E}_{\omega_5+3\omega_6+8\omega_7}$ on 
$E_7/P_1$.
\end{proposition}

\begin{proof}
Suppose that an irreducible equivariant vector bundle $\mathcal E_{\omega}$ on the $E_7$-adjoint variety $E_7/P_1$ is Ulrich. 
Then $a=0$ by Proposition \ref{coefficient of weight}. 
Since the decomposition of $\mathfrak e_7$ associated to $\alpha_1$ is a contact gradation and $\Phi_2 = \{ \theta \}$,  
we get $t = 1 + \sum_{i \neq 1} \frac{c_i (a_i +1)}{2} = b + \frac{3}{2}c + 2d + \frac{3}{2}e + f + \frac{1}{2}g + \frac{17}{2} >2$ for $\alpha \in \Phi_2$
and $c=0$ by Proposition~\ref{subminimum}. 
By straightforward computations considering the 33 positive roots in $\Phi^+ \setminus \Phi_0$, 
we get the set of singular values 


$ \Sing(\omega)=\{
1,
2,
d + 3,
b + d + 4,
d + e + 4,
b + d + e + 5,
d + e + f + 5, 
b + 2d + e + 6,
b + d + e + f + 6,
d + e + f + g + 6,
b + 2d + e + 7,
b + 2d + e + f + 7,
b + d + e + f + g + 7,
b + 2d + e + f + 8,
b + 2d + 2e + f + 8,
b + 2d + e + f + g + 8,
b + 2d + \frac{3}{2}e + f + \frac{1}{2}g + \frac{17}{2},
b + 2d + 2e + f + 9,
b + 2d + e + f + g + 9,
b + 2d + 2e + f + g + 9,
b + 3d + 2e + f + 10,
b + 2d + 2e + f + g + 10,
b + 2d + 2e + 2f + g + 10,
2b + 3d + 2e + f + 11,
b + 3d + 2e + f + g + 11,
b + 2d + 2e + 2f + g + 11,
2b + 3d + 2e + f + g + 12,
b + 3d + 2e + 2f + g + 12,
b + 3d + 3e + 2f + g + 13,
2b + 3d + 2e + 2f + g + 13,
2b + 3d + 3e + 2f + g + 14,
2b + 4d + 3e + 2f + g + 15,
2b + 4d + 3e + 2f + g + 16
\}$. 

By Lemma \ref{Ulrich criterion}, we want to find $b,d,e,f,g$ such that $\Sing(\omega) = \{1, 2, \cdots, 33\}$, hence $d=0$ from $d+3=3$. 
Because
$b + 2d + e + f + g + 9$ and 
$b + 2d + 2e + f + g + 9$ 
are different, we see that $e \neq 0$ 
so that $b=0$, $e=1$ from $b + d + 4 = 4$, $d + e + 4 = 5$. 
This implies that $d + e + f + 5 = 9$ and $2b + 4d + 3e + 2f + g + 16 = 33$. 
Consequently, we have $f=3$, $g=8$
and we can check that $\Sing(\omega_5+3\omega_6+8\omega_7)=\{ 1, 2, \cdots, 33 \}.$
Therefore, there is a unique irreducible equivariant Ulrich bundle $\mathcal{E}_{\omega_5+3\omega_6+8\omega_7}$ on the rational homogeneous variety $E_7/P_1$.
\end{proof}

\begin{remark} 
Because the dimension of the irreducible $\Spin(12)$-module with highest weight $8\varpi_1+3\varpi_2+\varpi_3$ is 
$2^7 \times 3^4 \times 5 \times 13 \times 17^2 \times 19$ by Weyl's character formula, 
the unique irreducible equivariant Ulrich bundle $\mathcal E_{\omega_5+3\omega_6+8\omega_7}$ on the $E_7$-adjoint variety $E_7/P_1$ has rank 3,700,494,720.
\end{remark} 

From Proposition \ref{E_6/P_1}, we already know the Cayley plane $E_6/P_1$ admits an irreducible equivariant Ulrich bundle. 
On the other hand, the existence of Ulrich bundles on the Cayley plane is a direct consequence of Proposition \ref{E_7/P_1} 
from the fact that the Cayley plane $E_6/P_1$ is a (maximal) smooth Schubert variety of 
$E_7/P_1$.

\begin{proposition}\label{Schubert variety}
Let $G/P$ be a rational homogeneous variety and $S \subset G/P$ a smooth Schubert variety of $G/P$. 
If $G/P$ admits an Ulrich bundle, then $S$ also admits an Ulrich bundle.
\end{proposition}

\begin{proof}
This result follows from the facts that a Schubert variety $S$ is an irreducible linear section of $G/P$ 
and the restriction of an Ulrich bundle to a linear section is also an Ulrich bundle from \cite{ES} and \cite{Be17}.
\end{proof}

\begin{corollary}
Let $E$ be the irreducible equivariant Ulrich bundle on the adjoint variety $E_7/P_1$ in Proposition \ref{E_7/P_1}.
Then the restriction $E_{\mid E_6/P_1}$ is an ($E_6$-equivariant) Ulrich bundle on the Cayley plane $E_6/P_1$.
\end{corollary}

\begin{proof}
Recall the characterization result on smooth Schubert varieties in a rational homogeneous variety $G/P$ with Picard number 1. 
A marked subdiagram of the marked Dynkin diagram of $G/P$ defines
a homogeneous subvariety $G_0/P_0$ of $G/P$, the $G_0$-orbit of the base point $eP \in G/P$,
which is a smooth Schubert variety; see Section 2 of \cite{HoM2}.
Conversely, we know that when a rational homogeneous variety $G/P$ is associated to a long simple root,
all smooth Schubert varieties are homogeneous subvarieties associated to subdiagrams of the marked Dynkin diagram by Proposition 3.7 of \cite{HoM2}.
Hence the Cayley plane $E_6/P_1$ is isomorphic to a (maximal) smooth Schubert variety of  the $E_7$-adjoint variety $E_7/P_1$. 
By Propositions \ref{E_7/P_1} and \ref{Schubert variety}, we deduce the existence of an Ulrich bundle on $E_6/P_1$.
\end{proof}

From Proposition \ref{Schubert variety}, we also obtain the existence of Ulrich bundles on odd symplectic Grassmannians $\Gr_{\omega}(2, 2n+1)$ of planes.
Let $V$ be a complex vector space endowed with a skew-symmetric bilinear form $\omega$ of maximal rank.
Denote the variety of all $k$-dimensional isotropic subspaces in $V$ by
$\Gr_{\omega}(k, V)=\{ W \subset V : \dim W=k, \, \omega|_{W} \equiv 0 \}.$
When $\dim V$ is even, say $\dim V = 2n$, the form $\omega$ is a nondegenerate symplectic form
and this variety $\Gr_{\omega}(k, 2n)$ is the usual symplectic Grassmannian,
which is homogeneous under the action of the symplectic group $\Sp(2n)$.
But when $\dim V$ is odd, say $\dim  V = 2n+1$,
the skew-form $\omega$ has a one-dimensional kernel.
The variety $\Gr_{\omega}(k, 2n+1)$,
called the \emph{odd symplectic Grassmannian},
is not homogeneous and has two orbits under the action of its automorphism group if $2 \leq k \leq n$; 
see \cite{Mihai} and \cite{Pasquier}.

\begin{corollary}
The odd symplectic Grassmannians of planes $\Gr_{\omega}(2, 2n+1)$ admit Ulrich bundles.
\end{corollary}

\begin{proof}
By the classification of smooth Schubert varieties in the symplectic Grassmannians (see \cite{Ho15}), 
an odd symplectic Grassmannian may be regarded as a smooth Schubert varieties in some symplectic Grassmannians. 
In fact, all nonhomogeneous smooth Schubert varieties in the symplectic Grassmannians are odd symplectic Grassmannians. 
Because the symplectic Grassmannians of planes admit irreducible equivariant Ulrich bundles by Proposition 3.5 of \cite{Fo}, 
this corollary is a direct consequence of Proposition \ref{Schubert variety}.
\end{proof}

\begin{remark}
In particular, if $k=2$, then the (odd) symplectic Grassmannians $\Gr_{\omega}(k, V)$ are general hyperplane sections of the Pl\"{u}cker embedding of $\Gr(k, V)$. 
For the Pl\"{u}cker embedding $\Gr(2, V) \hookrightarrow \mathbb P(\wedge^2 V)$,
a skew-symmetric 2-form $\omega$ on $V$ with maximal rank is a general element of $\wedge^2 V^*$.
Let $H \subset \wedge^2 V$ be the kernel of $\omega \in \wedge^2 V^*$. 
Then we get $\Gr(2, V) \cap H = \Gr_{\omega}(2, V)$. 
\end{remark}

\subsection{Homogeneous variety $E_7/P_2$}
The rational homogeneous variety $E_7/P_2 \subset \mathbb P^{911}$ has dimension 42 and Fano index 14. 

\begin{proposition}
There are no irreducible equivariant Ulrich bundles on 
$E_7/P_2$.
\end{proposition}

\begin{proof}
Suppose that an irreducible equivariant vector bundle $\mathcal E_{\omega}$ on $E_7/P_2$ is Ulrich. 
Then $b=0$ by Proposition \ref{coefficient of weight}. 
Since the Lie algebra $\mathfrak e_7$ has a gradation of depth 2 associated to $\alpha_2$ and 
$\alpha_1 + 2\alpha_2 + 2\alpha_3 + 3\alpha_4 + 2\alpha_5 + \alpha_6$ is minimal among the seven roots in $\Phi_2$, 
we know that the corresponding singular value $t = 1 + \sum_{i \neq 2} \frac{c_i (a_i +1)}{2}>5$ for all $\alpha \in \Phi_2$. 
Hence $d$ is zero by Proposition \ref{subminimum}. 

Considering the two positive roots 
$\alpha_1 + 2\alpha_2 + 2\alpha_3 + 3\alpha_4 + 3\alpha_5 + 2\alpha_6 + \alpha_7$ and 
$\alpha_1 + 2\alpha_2 + 2\alpha_3 + 4\alpha_4 + 3\alpha_5 + 2\alpha_6 + \alpha_7$ in $\Phi_2$, 
we check that 
$\{ 
\frac{1}{2}a + b + c + 2d + \frac{3}{2}e + f + \frac{1}{2}g + \frac{15}{2}, 
\frac{1}{2}a + b + c + \frac{3}{2}d + \frac{3}{2}e + f + \frac{1}{2}g + 7
\} \subset \Sing(\omega)$.
However, if $d=0$, then
these cannot be integers simultaneously. 
Therefore, 
there are no $a,b,c,d,e,f,g$ such that $\Sing(\omega)=\{ 1, 2, \cdots,42 \}$;  
so there are no irreducible equivariant Ulrich bundles on $E_7/P_2$ by Lemma \ref{Ulrich criterion}. 
\end{proof}

\subsection{Homogeneous variety $E_7/P_3$}
The rational homogeneous variety $E_7/P_3 \subset \mathbb P^{8644}$ has dimension 47 and Fano index 11. 

\begin{proposition}
There are no irreducible equivariant Ulrich bundles on 
$E_7/P_3$.
\end{proposition}

\begin{proof}
Suppose that an irreducible equivariant vector bundle $\mathcal E_{\omega}$ on $E_7/P_3$ is Ulrich. 
Then $c=0$ by Proposition \ref{coefficient of weight}. 
Because the Lie algebra $\mathfrak e_7$ has a gradation of depth 3 associated to $\alpha_3$ and 
$\alpha_1 + \alpha_2 + 2\alpha_3 + 2\alpha_4 + \alpha_5$, 
$\alpha_1 + 2\alpha_2 + 3\alpha_3 + 4\alpha_4 + 3\alpha_5 + 2\alpha_6 + \alpha_7$ are minimal among the 15 roots in $\Phi_2$ and the two roots in $\Phi_3$, respectively,  
we know that the corresponding singular values 
$t = 1 + \sum_{i \neq 3} \frac{c_i (a_i +1)}{2}>3$ for all $\alpha \in \Phi_2$ and 
$t = 1 + \sum_{i \neq 3} \frac{c_i (a_i +1)}{3}>5$ for all $\alpha \in \Phi_3$. 
Hence, either $a=0$ or $d=0$ by Proposition \ref{subminimum}. 

Considering the two positive roots 
$\alpha_1 + 2\alpha_2 + 3\alpha_3 + 4\alpha_4 + 3\alpha_5 + 2\alpha_6 + \alpha_7$  and 
$2\alpha_1 + 2\alpha_2 + 3\alpha_3 + 4\alpha_4 + 3\alpha_5 + 2\alpha_6 + \alpha_7$ in $\Phi_3$, 
we check that 
$\{ 
\frac{2}{3}a + \frac{2}{3}b + c + \frac{4}{3}d + e + \frac{2}{3}f + \frac{1}{3}g + \frac{17}{3},
\frac{1}{3}a + \frac{2}{3}b + c + \frac{4}{3}d + e + \frac{2}{3}f + \frac{1}{3}g + \frac{16}{3}
\} \subset \Sing(\omega)$.
However, if $a=0$, then these cannot be integers simultaneously. 
Likewise, considering the two positive roots 
$\alpha_1 + 2\alpha_2 + 2\alpha_3 + 3\alpha_4 + 3\alpha_5 + 2\alpha_6 + \alpha_7$  and 
$\alpha_1 + 2\alpha_2 + 2\alpha_3 + 4\alpha_4 + 3\alpha_5 + 2\alpha_6 + \alpha_7$ in $\Phi_2$, 
we check that 
$\{ 
\frac{1}{2}a + b + c + 2d + \frac{3}{2}e + f + \frac{1}{2}g + \frac{15}{2},
\frac{1}{2}a + b + c + \frac{3}{2}d + \frac{3}{2}e + f + \frac{1}{2}g + 7
\} \subset \Sing(\omega)$.
However, if $d=0$, then these cannot be integers simultaneously. 
Therefore, 
there are no $a,b,c,d,e,f,g$ such that $\Sing(\omega)=\{ 1, 2, \cdots,47 \}.$
\end{proof}

\subsection{Homogeneous variety $E_7/P_4$}
The rational homogeneous variety $E_7/P_4 \subset \mathbb P^{365749}$ has dimension 53 and Fano index 8. 

\begin{proposition}
There are no irreducible equivariant Ulrich bundles on 
$E_7/P_4$.
\end{proposition}

\begin{proof}
Suppose that an irreducible equivariant vector bundle $\mathcal E_{\omega}$ on $E_7/P_4$ is Ulrich. 
Then $d=0$ by Proposition \ref{coefficient of weight}. 
Because the Lie algebra $\mathfrak e_7$ has a gradation of depth 4 associated to $\alpha_4$ and 
$\alpha_2 + \alpha_3 + 2\alpha_4 + \alpha_5$, 
$\alpha_1 + \alpha_2 + 2\alpha_3 + 3\alpha_4 + 2\alpha_5 + \alpha_6$, 
$\alpha_1 + 2\alpha_2 + 2\alpha_3 + 4\alpha_4 + 3\alpha_5 + 2\alpha_6 + \alpha_7$ 
are minimal among the 18 roots in $\Phi_2$, the 8 roots in $\Phi_3$ and the 3 roots in $\Phi_4$, respectively,  
we know that the corresponding singular values 
$t = 1 + \sum_{i \neq 4} \frac{c_i (a_i +1)}{2}>2$ for all $\alpha \in \Phi_2$, 
$t = 1 + \sum_{i \neq 4} \frac{c_i (a_i +1)}{3}>3$ for all $\alpha \in \Phi_3$. and 
$t = 1 + \sum_{i \neq 4} \frac{c_i (a_i +1)}{4}>3$ for all $\alpha \in \Phi_4$. 
Hence at least one of $b, c, e$ is zero by Proposition \ref{subminimum}. 

Considering the two positive roots 
$\alpha_1 + \alpha_2 + 2\alpha_3 + 3\alpha_4 + 3\alpha_5 + 2\alpha_6 + \alpha_7$  and 
$\alpha_1 + 2\alpha_2 + 2\alpha_3 + 3\alpha_4 + 3\alpha_5 + 2\alpha_6 + \alpha_7$  in $\Phi_3$, 
we check that 
$\{ 
\frac{1}{3}a + \frac{1}{3}b + \frac{2}{3}c + d + \frac{2}{3}e + \frac{1}{3}f + \frac{10}{3}, 
\frac{1}{3}a + \frac{2}{3}b + \frac{2}{3}c + d + \frac{2}{3}e + \frac{1}{3}f + \frac{11}{3}
\} \subset \Sing(\omega)$.
However, if $b=0$, then these cannot be integers simultaneously. 
Likewise, considering the two positive roots 
$\alpha_1 + 2\alpha_2 + 2\alpha_3 + 4\alpha_4 + 3\alpha_5 + 2\alpha_6 + \alpha_7$ and 
$\alpha_1 + 2\alpha_2 + 3\alpha_3 + 4\alpha_4 + 3\alpha_5 + 2\alpha_6 + \alpha_7$ in $\Phi_4$, 
we check that 
$\{ 
\frac{1}{4}a + \frac{1}{2}b + \frac{1}{2}c + d + \frac{3}{4}e + \frac{1}{2}f + \frac{1}{4}g + \frac{15}{4}, 
\frac{1}{4}a + \frac{1}{2}b + \frac{3}{4}c + d + \frac{3}{4}e + \frac{1}{2}f + \frac{1}{4}g + 4
\} \subset \Sing(\omega)$.
However, if $c=0$, then these cannot be integers simultaneously. 
Considering the two positive roots 
$\alpha_2 + \alpha_3 + 2\alpha_4 + \alpha_5 + \alpha_6$ and 
$\alpha_2 + \alpha_3 + 2\alpha_4 + 2\alpha_5 + \alpha_6$  in $\Phi_2$, 
we check that 
$\{ 
\frac{1}{2}b + \frac{1}{2}c + d + \frac{1}{2}e + \frac{1}{2}f + 3, 
\frac{1}{2}b + \frac{1}{2}c + d + e + \frac{1}{2}f + \frac{7}{2}
\} \subset \Sing(\omega)$.
However, if $e=0$, then these cannot be integers simultaneously. 
Therefore, 
there are no $a,b,c,d,e,f,g$ such that $\Sing(\omega)=\{ 1, 2, \cdots,53 \}.$
\end{proof}

\subsection{Homogeneous variety $E_7/P_5$}
The rational homogeneous variety $E_7/P_5 \subset \mathbb P^{27663}$ has dimension 50 and Fano index 10. 

\begin{proposition}
There are no irreducible equivariant Ulrich bundles on 
$E_7/P_5$.
\end{proposition}

\begin{proof}
Suppose that an irreducible equivariant vector bundle $\mathcal E_{\omega}$ on $E_7/P_5$ is Ulrich. 
Then $e=0$ by Proposition~\ref{coefficient of weight}. 
Because the Lie algebra $\mathfrak e_7$ has a gradation of depth 3 associated to $\alpha_5$ and 
$\alpha_2 + \alpha_3 + 2\alpha_4 + 2\alpha_5 + \alpha_6$, 
$\alpha_1 + \alpha_2 + 2\alpha_3 + 3\alpha_4 + 3\alpha_5 + 2\alpha_6 + \alpha_7$ 
are minimal among the 15 roots in $\Phi_2$ and the 5 roots in $\Phi_3$, respectively,  
we know that the corresponding singular values 
$t = 1 + \sum_{i \neq 5} \frac{c_i (a_i +1)}{2}>3$ for all $\alpha \in \Phi_2$ and  
$t = 1 + \sum_{i \neq 5} \frac{c_i (a_i +1)}{3}>4$ for all $\alpha \in \Phi_3$.  
Hence either $d=0$ or $f=0$ by Proposition \ref{subminimum}. 

Considering the two positive roots 
$\alpha_1 + 2\alpha_2 + 2\alpha_3 + 3\alpha_4 + 3\alpha_5 + 2\alpha_6 + \alpha_7$ and 
$\alpha_1 + 2\alpha_2 + 2\alpha_3 + 4\alpha_4 + 3\alpha_5 + 2\alpha_6 + \alpha_7$ in $\Phi_3$, 
we check that 
$\{ 
\frac{1}{3}a + \frac{2}{3}b + \frac{2}{3}c + d + e + \frac{2}{3}f + \frac{1}{3}g + \frac{14}{3}, 
\frac{1}{3}a + \frac{2}{3}b + \frac{2}{3}c + \frac{4}{3}d + e + \frac{2}{3}f + \frac{1}{3}g + 5
\} \subset \Sing(\omega)$.
However, if $d=0$, then these cannot be integers simultaneously. 
Likewise, considering the two positive roots 
$\alpha_1 + 2\alpha_2 + 2\alpha_3 + 3\alpha_4 + 2\alpha_5 + \alpha_6 + \alpha_7$ and 
$\alpha_1 + 2\alpha_2 + 2\alpha_3 + 3\alpha_4 + 2\alpha_5 + 2\alpha_6 + \alpha_7$ in $\Phi_2$, 
we check that 
$\{ 
\frac{1}{2}a + b + c + \frac{3}{2}d + e + \frac{1}{2}f + \frac{1}{2}g + 6, 
\frac{1}{2}a + b + c + \frac{3}{2}d + e + f + \frac{1}{2}g + \frac{13}{2}
\} \subset \Sing(\omega)$.
However, if $f=0$, then these cannot be integers simultaneously. 
Therefore, 
there are no $a,b,c,d,e,f,g$ such that $\Sing(\omega)=\{ 1, 2, \cdots,50 \}.$
\end{proof}

\subsection{Homogeneous variety $E_7/P_6$}
The rational homogeneous variety $E_7/P_6 \subset \mathbb P^{1538}$ has dimension 42 and Fano index 13. 

\begin{proposition} 
There are no irreducible equivariant Ulrich bundles on 
$E_7/P_6$.
\end{proposition}

\begin{proof}
Suppose that an irreducible equivariant vector bundle $\mathcal E_{\omega}$ on $E_7/P_6$ is Ulrich. 
Then $f=0$ by Proposition \ref{coefficient of weight}. 
Because the Lie algebra $\mathfrak e_7$ has a gradation of depth 2 associated to $\alpha_6$ and 
$\alpha_2 + \alpha_3 + 2\alpha_4 + 2\alpha_5 + 2\alpha_6 + \alpha_7$, 
is minimal among the 10 roots in $\Phi_2$, 
we know that the corresponding singular value 
$t = 1 + \sum_{i \neq 6} \frac{c_i (a_i +1)}{2}>4$ for all $\alpha \in \Phi_2$.
Hence either $e=0$ or $g=0$ by Proposition \ref{subminimum}. 

Considering the two positive roots 
$\alpha_1 + 2\alpha_2 + 2\alpha_3 + 3\alpha_4 + 2\alpha_5 + 2\alpha_6 + \alpha_7$ and 
$\alpha_1 + 2\alpha_2 + 2\alpha_3 + 3\alpha_4 + 3\alpha_5 + 2\alpha_6 + \alpha_7$ in $\Phi_2$, 
we check that 
$\{ 
\frac{1}{2}a + b + c + \frac{3}{2}d + e + f + \frac{1}{2}g + \frac{13}{2}, 
\frac{1}{2}a + b + c + \frac{3}{2}d + \frac{3}{2}e + f + \frac{1}{2}g + 7
\} \subset \Sing(\omega)$.
However, if $e=0$, then these cannot be integers simultaneously. 
Thus $g=0$ and $e$ is an odd integer. 
Likewise, considering the two positive roots 
$\alpha_1 + 2\alpha_2 + 2\alpha_3 + 4\alpha_4 + 3\alpha_5 + 2\alpha_6 + \alpha_7$ and 
$\alpha_1 + 2\alpha_2 + 3\alpha_3 + 4\alpha_4 + 3\alpha_5 + 2\alpha_6 + \alpha_7$ in $\Phi_2$, 
we check that 
$\{ 
\frac{1}{2}a + b + c + 2d + \frac{3}{2}e + f + \frac{1}{2}g + \frac{15}{2}, 
\frac{1}{2}a + b + \frac{3}{2}c + 2d + \frac{3}{2}e + f + \frac{1}{2}g + 8
\} \subset \Sing(\omega)$.
Then $c$ must be an odd integer. 
However, if $c$ is odd, then a singular value $a + b + \frac{3}{2}c + 2d + \frac{3}{2}e + f + \frac{1}{2}g + \frac{17}{2}$ 
attained at a root $2\alpha_1 + 2\alpha_2 + 3\alpha_3 + 4\alpha_4 + 3\alpha_5 + 2\alpha_6 + \alpha_7$
cannot be an integer. 
Therefore, 
there are no $a,b,c,d,e,f,g$ such that $\Sing(\omega)=\{ 1, 2, \cdots, 42 \}.$
\end{proof}

\subsection{Homogeneous variety $E_7/P_7$}
The rational homogeneous variety $E_7/P_7 \subset \mathbb P^{55}$ has dimension 27 and Fano index 18. 
This is usually called the \emph{Freudenthal variety}, which is one of the two exceptional Hermitian symmetric spaces of compact type.

\begin{proposition}
There are no irreducible equivariant Ulrich bundles on 
$E_7/P_7$.
\end{proposition}

\begin{proof}
Suppose that an irreducible equivariant vector bundle $\mathcal E_{\omega}$ on $E_7/P_7$ is Ulrich. 
Then $g=0$ by Proposition \ref{coefficient of weight} and 
$f=0$ by Proposition \ref{subminimum} because the Lie algebra $\mathfrak e_7$ has a gradation of depth 1 associated to $\alpha_7$.
Since $\alpha_5 + \alpha_6 + \alpha_7$ is minimal among all roots in $\Phi_1 \setminus \{ \alpha_7 , \alpha_6 + \alpha_7\}$, 
the corresponding singular value $e + f + g + 3$ is equal to 3, which implies that $e=0$.
Likewise, since $\alpha_4 + \alpha_5 + \alpha_6 + \alpha_7$ is minimal among all roots in $\Phi_1 \setminus \{ \alpha_7 , \alpha_6 + \alpha_7, \alpha_5 + \alpha_6 + \alpha_7 \}$, 
the corresponding singular value $d+ e + f + g + 4$ is equal to 4, which implies that  $d=0$.

According to Proposition \ref{maximum of singular values}, 
a singular value attained at the highest root $\theta = 2\alpha_1 + 2\alpha_2 + 3\alpha_3 + 4\alpha_4 + 3\alpha_5 + 2\alpha_6 + \alpha_7$ 
is the maximum of $\Sing(\omega)$, from which we see that $2a + 2b + 3c + 4d + 3e + 2f + g + 17=27$. 
As $\theta - \alpha_7$ is maximal among all roots in $\Phi_1 \setminus \{ \theta \}$, we get that $a + 2b + 3c + 4d + 3e + 2f + g + 17=26$. 
Hence we deduce that $a=0$. 
On the other hand, considering the two positive roots 
$\alpha_1 + \alpha_2 + \alpha_3 + \alpha_4 + \alpha_5 + \alpha_6 + \alpha_7$ and 
$\alpha_2 + \alpha_3 + 2\alpha_4 + \alpha_5 + \alpha_6 + \alpha_7$ in $\Phi_1$, 
we check that 
$\{ 
a + b + c + d + e + f + g + 7,
b + c + 2d + e + f + g + 7
\} \subset \Sing(\omega)$.
However, if $a=0$, then these cannot be different. 
%
%
So there are no $a,b,c,d,e,f,g$ such that $\Sing(\omega)=\{ 1, 2, \cdots, 27 \}$. 
Therefore, 
$E_7/P_7$ does not admit an irreducible equivariant Ulrich bundle. 
\end{proof}
%

\section{$E_8$-homogeneous varieties}
We recall some basic facts about the simple Lie group $E_8$. 
When we choose an orthonormal basis $\{L_1, \cdots, L_8\}$ for the dual Cartan subalgebra $\mathfrak t^*$, 
the simple roots can be taken as follows: \\ 
$\alpha_1 = \frac{1}{2} (L_1-L_2-L_3-L_4-L_5-L_6-L_7+L_8)$,  
$\alpha_2 = L_1 + L_2$, $\alpha_3 = L_2 - L_1$, $\alpha_4 = L_3 - L_2$, 
$\alpha_5 = L_4 - L_3$, $\alpha_6 = L_5 - L_4$, $\alpha_7 = L_6 - L_5$, $\alpha_8 = L_7 - L_6$.

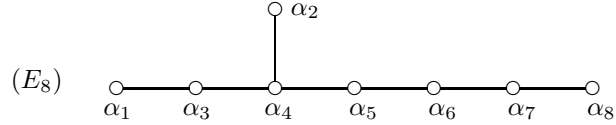
\begin{center}
\begin{picture} (250, 55)
 \put(10,20){$(E_8)$}

 \put(50,20){\circle{5}} 
 \put(110,50){\circle{5}} 
 \put(80,20){\circle{5}} 
 \put(110,20){\circle{5}} 
 \put(140,20){\circle{5}} 
 \put(170,20){\circle{5}} 
 \put(200,20){\circle{5}} 
 \put(230,20){\circle{5}} 

 \put(52.5,20){\line(1,0){25}} 
 \put(82.5,20){\line(1,0){25}} 
 \put(112.5,20){\line(1,0){25}} 
 \put(142.5,20){\line(1,0){25}} 
 \put(110,22.5){\line(0,1){25}} 
 \put(172.5,20){\line(1,0){25}} 
 \put(202.5,20){\line(1,0){25}} 

 \put(45,10){$\alpha_1$} 
 \put(116,48){$\alpha_2$} 
 \put(75,10){$\alpha_3$} 
 \put(106,10){$\alpha_4$} 
 \put(138,10){$\alpha_5$} 
 \put(168,10){$\alpha_6$} 
 \put(198,10){$\alpha_7$} 
 \put(228,10){$\alpha_8$} 
 \end{picture}
\end{center}
The Lie group $E_8$ has 120 positive roots, see e.g. page 333 in \cite{FH}:
$$\Phi^+=\{ L_i+L_j\}_{1 \leq i < j \leq 8} \cup \{ L_i-L_j\}_{1 \leq j < i \leq 8} 
\cup \Big \{ \frac{1}{2} \Big (\sum_{i=1}^{7} (-1)^{n(i)}L_i + L_8 \Big ) : \sum_{i=1}^{7} n(i) \mbox{ is even} \Big \}.$$
The highest root of $E_8$ is $2\alpha_1 + 3\alpha_2 + 4\alpha_3 + 6\alpha_4 + 5\alpha_5 + 4\alpha_6 + 3\alpha_7 + 2\alpha_8$.


\subsection{Homogeneous variety $E_8/P_1$}
The rational homogeneous variety $E_8/P_1 \subset \mathbb P^{3874}$ has dimension 78 and Fano index 23. 

\begin{proposition}
There are no irreducible equivariant Ulrich bundles on 
$E_8/P_1$.
\end{proposition}

\begin{proof} 
Suppose that an irreducible equivariant vector bundle $\mathcal E_{\omega}$ on $E_8/P_1$ is Ulrich. 
Then $a=0$ by Proposition~\ref{coefficient of weight}. 
Because the Lie algebra $\mathfrak e_8$ has a gradation of depth 2 associated to $\alpha_1$ and 
$2\alpha_1 + 2\alpha_2 + 3\alpha_3 + 4\alpha_4 + 3\alpha_5 + 2\alpha_6 + \alpha_7$ is minimal among the 14 roots in $\Phi_2$, 
we know that the corresponding singular value $t = 1 + \sum_{i \neq 1} \frac{c_i (a_i +1)}{2}>8$ for all $\alpha \in \Phi_2$. 
Hence $c$ is zero by Proposition \ref{subminimum}. 

Considering the two positive roots 
$2\alpha_1 + 3\alpha_2 + 3\alpha_3 + 5\alpha_4 + 4\alpha_5 + 3\alpha_6 + 2\alpha_7 + \alpha_8$ and 
$2\alpha_1 + 3\alpha_2 + 4\alpha_3 + 5\alpha_4 + 4\alpha_5 + 3\alpha_6 + 2\alpha_7 + \alpha_8$ in $\Phi_2$, 
we check that 
$\{ 
a + \frac{3}{2}b + \frac{3}{2}c + \frac{5}{2}d + 2e + \frac{3}{2}f + g + \frac{1}{2}h + \frac{23}{2}, 
a + \frac{3}{2}b + 2c + \frac{5}{2}d + 2e + \frac{3}{2}f + g + \frac{1}{2}h + 12
\} \subset \Sing(\omega)$.
However, if $c=0$, then
these cannot be integers simultaneously. 
Therefore, 
there are no $a,b,c,d,e,f,g,h$ such that $\Sing(\omega)=\{ 1, 2, \cdots,78 \}$; 
so there are no irreducible equivariant Ulrich bundles on $E_8/P_1$ by Lemma \ref{Ulrich criterion}. 
\end{proof}

\subsection{Homogeneous variety $E_8/P_2$}
The rational homogeneous variety $E_8/P_2 \subset \mathbb P^{147249}$ has dimension 92 and Fano index 17. 

\begin{proposition}
There are no irreducible equivariant Ulrich bundles on 
$E_8/P_2$.
\end{proposition}

\begin{proof}
Suppose that an irreducible equivariant vector bundle $\mathcal E_{\omega}$ on $E_8/P_2$ is Ulrich. 
Then $b=0$ by Proposition~\ref{coefficient of weight}. 
Because the Lie algebra $\mathfrak e_8$ has a gradation of depth 3 associated to $\alpha_2$ and 
$\alpha_1 + 2\alpha_2 + 2\alpha_3 + 3\alpha_4 + 2\alpha_5 + \alpha_6$, 
$\alpha_1 + 3\alpha_2 + 3\alpha_3 + 5\alpha_4 + 4\alpha_5 + 3\alpha_6 + 2\alpha_7 + \alpha_8$ are minimal among the 28 roots in $\Phi_2$ and the 8 roots in $\Phi_3$, respectively,  
we know that the corresponding singular values 
$t = 1 + \sum_{i \neq 2} \frac{c_i (a_i +1)}{2}>5$ for all $\alpha \in \Phi_2$ and 
$t = 1 + \sum_{i \neq 2} \frac{c_i (a_i +1)}{3}>7$ for all $\alpha \in \Phi_3$. 
Hence $d$ is zero by Proposition \ref{subminimum}. 

Considering the two positive roots 
$2\alpha_1 + 3\alpha_2 + 4\alpha_3 + 5\alpha_4 + 4\alpha_5 + 3\alpha_6 + 2\alpha_7 + \alpha_8$  and 
$2\alpha_1 + 3\alpha_2 + 4\alpha_3 + 6\alpha_4 + 4\alpha_5 + 3\alpha_6 + 2\alpha_7 + \alpha_8$ in $\Phi_3$, 
we check that 
$\{ 
\frac{2}{3}a + b + \frac{4}{3}c + \frac{5}{3}d + \frac{4}{3}e + f + \frac{2}{3}g + \frac{1}{3}h + 8, 
\frac{2}{3}a + b + \frac{4}{3}c + 2d + \frac{4}{3}e + f + \frac{2}{3}g + \frac{1}{3}h + \frac{25}{3}
\} \subset \Sing(\omega)$.
However, if $d=0$, then these cannot be integers simultaneously. 
Therefore, 
there are no $a,b,c,d,e,f,g,h$ such that $\Sing(\omega)=\{ 1, 2, \cdots, 92 \}.$
\end{proof}

\subsection{Homogeneous variety $E_8/P_3$}
The rational homogeneous variety $E_8/P_3 \subset \mathbb P^{6695999}$ has dimension 98 and Fano index 13. 

\begin{proposition}
There are no irreducible equivariant Ulrich bundles on 
$E_8/P_3$.
\end{proposition}

\begin{proof}
Suppose that an irreducible equivariant vector bundle $\mathcal E_{\omega}$ on $E_8/P_3$ is Ulrich. 
Then $c=0$ by Proposition~\ref{coefficient of weight}. 
Because the Lie algebra $\mathfrak e_8$ has a gradation of depth 4 associated to $\alpha_3$ and 
$\alpha_1 + \alpha_2 + 2\alpha_3 + 2\alpha_4 + \alpha_5$, 
$\alpha_1 + 2\alpha_2 + 3\alpha_3 + 4\alpha_4 + 3\alpha_5 + 2\alpha_6 + \alpha_7$, 
$2\alpha_1 + 2\alpha_2 + 4\alpha_3 + 5\alpha_4 + 4\alpha_5 + 3\alpha_6 + 2\alpha_7 + \alpha_8$ are minimal among the 35 roots in $\Phi_2$, the 14 roots in $\Phi_3$ and the 7 roots in $\Phi_4$, respectively,  
we know that the corresponding singular values 
$t = 1 + \sum_{i \neq 3} \frac{c_i (a_i +1)}{2}>3$ for all $\alpha \in \Phi_2$, 
$t = 1 + \sum_{i \neq 3} \frac{c_i (a_i +1)}{3}>5$ for all $\alpha \in \Phi_3$ and 
$t = 1 + \sum_{i \neq 3} \frac{c_i (a_i +1)}{4}>5$ for all $\alpha \in \Phi_4$. 
Hence either $a=0$ or $d=0$ by Proposition \ref{subminimum}.

Considering the two positive roots 
$\alpha_1 + 2\alpha_2 + 3\alpha_3 + 5\alpha_4 + 4\alpha_5 + 3\alpha_6 + 2\alpha_7 + \alpha_8$ and 
$2\alpha_1 + 2\alpha_2 + 3\alpha_3 + 5\alpha_4 + 4\alpha_5 + 3\alpha_6 + 2\alpha_7 + \alpha_8$ in $\Phi_3$, 
we check that 
$\{ 
\frac{1}{3}a + b + c + \frac{5}{3}d + \frac{4}{3}e + f + \frac{2}{3}g + \frac{1}{3}h + \frac{22}{3}, 
\frac{2}{3}a + b + c + \frac{5}{3}d + \frac{4}{3}e + f + \frac{2}{3}g + \frac{1}{3}h + \frac{23}{3}
\} \subset \Sing(\omega)$.
However, if $a=0$, then these cannot be integers simultaneously. 
Likewise, considering the two positive roots 
$2\alpha_1 + 3\alpha_2 + 4\alpha_3 + 5\alpha_4 + 4\alpha_5 + 3\alpha_6 + 2\alpha_7 + \alpha_8$ and 
$2\alpha_1 + 3\alpha_2 + 4\alpha_3 + 6\alpha_4 + 4\alpha_5 + 3\alpha_6 + 2\alpha_7 + \alpha_8$ in $\Phi_4$, 
we check that 
$\{ 
\frac{1}{2}a + \frac{3}{4}b + c + \frac{5}{4}d + e + \frac{3}{4}f + \frac{1}{2}g + \frac{1}{4}h + 6, 
\frac{1}{2}a + \frac{3}{4}b + c + \frac{3}{2}d + e + \frac{3}{4}f + \frac{1}{2}g + \frac{1}{4}h + \frac{25}{4}
\} \subset \Sing(\omega)$.
However, if $d=0$, then these cannot be integers simultaneously. 
Therefore, 
there are no $a,b,c,d,e,f,g,h$ such that $\Sing(\omega)=\{ 1, 2, \cdots, 98 \}.$
\end{proof}

\subsection{Homogeneous variety $E_8/P_4$}
The rational homogeneous variety $E_8/P_4 \subset \mathbb P^{6899079263}$ has dimension 106 and Fano index 9. 

\begin{proposition}
There are no irreducible equivariant Ulrich bundles on 
$E_8/P_4$.
\end{proposition}

\begin{proof}
Suppose that an irreducible equivariant vector bundle $\mathcal E_{\omega}$ on $E_8/P_4$ is Ulrich. 
Then $d=0$ by Proposition~\ref{coefficient of weight}. 
Because the Lie algebra $\mathfrak e_8$ has a gradation of depth 6 associated to $\alpha_4$ and 
$\alpha_2 + \alpha_3 + 2\alpha_4 + \alpha_5$, 
$\alpha_1 + \alpha_2 + 2\alpha_3 + 3\alpha_4 + 2\alpha_5 + \alpha_6$, 
$\alpha_1 + 2\alpha_2 + 2\alpha_3 + 4\alpha_4 + 3\alpha_5 + 2\alpha_6 + \alpha_7$, 
$\alpha_1 + 2\alpha_2 + 3\alpha_3 + 5\alpha_4 + 4\alpha_5 + 3\alpha_6 + 2\alpha_7 + \alpha_8$, 
$2\alpha_1 + 3\alpha_2 + 4\alpha_3 + 6\alpha_4 + 4\alpha_5 + 3\alpha_6 + 2\alpha_7 + \alpha_8$ 
are minimal among the 30 roots in $\Phi_2$, the 20 roots in $\Phi_3$, the 15 roots in $\Phi_4$, the 6 roots in $\Phi_5$ and the 5 roots in $\Phi_6$, respectively,  
we know that the corresponding singular values 
$t = 1 + \sum_{i \neq 4} \frac{c_i (a_i +1)}{2}>2$ for all $\alpha \in \Phi_2$, 
$t = 1 + \sum_{i \neq 4} \frac{c_i (a_i +1)}{3}>3$ for all $\alpha \in \Phi_3$,
$t = 1 + \sum_{i \neq 4} \frac{c_i (a_i +1)}{4}>3$ for all $\alpha \in \Phi_4$,
$t = 1 + \sum_{i \neq 4} \frac{c_i (a_i +1)}{5}>4$ for all $\alpha \in \Phi_5$ and 
$t = 1 + \sum_{i \neq 4} \frac{c_i (a_i +1)}{6}>4$ for all $\alpha \in \Phi_6$. 
Hence at least one of $b, c, e$ is zero by Proposition \ref{subminimum}.

Considering the two positive roots 
$\alpha_1 + 2\alpha_2 + 3\alpha_3 + 5\alpha_4 + 4\alpha_5 + 3\alpha_6 + 2\alpha_7 + \alpha_8$  and 
$\alpha_1 + 3\alpha_2 + 3\alpha_3 + 5\alpha_4 + 4\alpha_5 + 3\alpha_6 + 2\alpha_7 + \alpha_8$  in $\Phi_5$, 
we check that 
$\{ 
\frac{1}{5}a + \frac{2}{5}b + \frac{3}{5}c + d + \frac{4}{5}e + \frac{3}{5}f + \frac{2}{5}g + \frac{1}{5}h + \frac{21}{5}, 
\frac{1}{5}a + \frac{3}{5}b + \frac{3}{5}c + d + \frac{4}{5}e + \frac{3}{5}f + \frac{2}{5}g + \frac{1}{5}h + \frac{22}{5}
\} \subset \Sing(\omega)$.
However, if $b=0$, then these cannot be integers simultaneously. 
Likewise, considering the two positive roots 
$2\alpha_1 + 3\alpha_2 + 3\alpha_3 + 5\alpha_4 + 4\alpha_5 + 3\alpha_6 + 2\alpha_7 + \alpha_8$  and 
$2\alpha_1 + 3\alpha_2 + 4\alpha_3 + 5\alpha_4 + 4\alpha_5 + 3\alpha_6 + 2\alpha_7 + \alpha_8$  in $\Phi_5$, 
we check that 
$\{ 
\frac{2}{5}a + \frac{3}{5}b + \frac{3}{5}c + d + \frac{4}{5}e + \frac{3}{5}f + \frac{2}{5}g + \frac{1}{5}h + \frac{23}{5}, 
\frac{2}{5}a + \frac{3}{5}b + \frac{4}{5}c + d + \frac{4}{5}e + \frac{3}{5}f + \frac{2}{5}g + \frac{1}{5}h + \frac{24}{5}
\} \subset \Sing(\omega)$.
However, if $c=0$, then these cannot be integers simultaneously. 
Considering the two positive roots 
$2\alpha_1 + 3\alpha_2 + 4\alpha_3 + 6\alpha_4 + 4\alpha_5 + 3\alpha_6 + 2\alpha_7 + \alpha_8$  and 
$2\alpha_1 + 3\alpha_2 + 4\alpha_3 + 6\alpha_4 + 5\alpha_5 + 3\alpha_6 + 2\alpha_7 + \alpha_8$  in $\Phi_6$, 
we check that 
$\{ 
\frac{1}{3}a + \frac{1}{2}b + \frac{2}{3}c + d + \frac{2}{3}e + \frac{1}{2}f + \frac{1}{3}g + \frac{1}{6}h + \frac{25}{6}, 
\frac{1}{3}a + \frac{1}{2}b + \frac{2}{3}c + d + \frac{5}{6}e + \frac{1}{2}f + \frac{1}{3}g + \frac{1}{6}h + \frac{13}{3}
\} \subset \Sing(\omega)$.
However, if $e=0$, then these cannot be integers simultaneously. 
Therefore, 
there are no $a,b,c,d,e,f,g,h$ such that $\Sing(\omega)=\{ 1, 2, \cdots, 106 \}.$
\end{proof}

\subsection{Homogeneous variety $E_8/P_5$}
The rational homogeneous variety $E_8/P_5 \subset \mathbb P^{146325269}$ has dimension 104 and Fano index 11. 

\begin{proposition}
There are no irreducible equivariant Ulrich bundles on 
$E_8/P_5$.
\end{proposition}

\begin{proof}
Suppose that an irreducible equivariant vector bundle $\mathcal E_{\omega}$ on $E_8/P_5$ is Ulrich. 
Then $e=0$ by Proposition~\ref{coefficient of weight}. 
Because the Lie algebra $\mathfrak e_8$ has a gradation of depth 5 associated to $\alpha_5$ and 
$\alpha_2 + \alpha_3 + 2\alpha_4 + 2\alpha_5 + \alpha_6$, 
$\alpha_1 + \alpha_2 + 2\alpha_3 + 3\alpha_4 + 3\alpha_5 + 2\alpha_6 + \alpha_7$, 
$\alpha_1 + 2\alpha_2 + 2\alpha_3 + 4\alpha_4 + 4\alpha_5 + 3\alpha_6 + 2\alpha_7 + \alpha_8$, 
$2\alpha_1 + 3\alpha_2 + 4\alpha_3 + 6\alpha_4 + 5\alpha_5 + 3\alpha_6 + 2\alpha_7 + \alpha_8$ 
are minimal among the 30 roots in $\Phi_2$, the 20 roots in $\Phi_3$, the 10 roots in $\Phi_4$ and the 4 roots in $\Phi_5$, respectively,  
we know that the corresponding singular values 
$t = 1 + \sum_{i \neq 5} \frac{c_i (a_i +1)}{2}>3$ for all $\alpha \in \Phi_2$, 
$t = 1 + \sum_{i \neq 5} \frac{c_i (a_i +1)}{3}>4$ for all $\alpha \in \Phi_3$,
$t = 1 + \sum_{i \neq 5} \frac{c_i (a_i +1)}{4}>4$ for all $\alpha \in \Phi_4$ and 
$t = 1 + \sum_{i \neq 5} \frac{c_i (a_i +1)}{5}>5$ for all $\alpha \in \Phi_5$. 
Hence either $d=0$ or $f=0$ by Proposition \ref{subminimum}.

Considering the two positive roots 
$2\alpha_1 + 3\alpha_2 + 4\alpha_3 + 5\alpha_4 + 4\alpha_5 + 3\alpha_6 + 2\alpha_7 + \alpha_8$ and 
$2\alpha_1 + 3\alpha_2 + 4\alpha_3 + 6\alpha_4 + 4\alpha_5 + 3\alpha_6 + 2\alpha_7 + \alpha_8$ in $\Phi_4$, 
we check that 
$\{ 
\frac{1}{2}a + \frac{3}{4}b + c + \frac{5}{4}d + e + \frac{3}{4}f + \frac{1}{2}g + \frac{1}{4}h + 6,
\frac{1}{2}a + \frac{3}{4}b + c + \frac{3}{2}d + e + \frac{3}{4}f + \frac{1}{2}g + \frac{1}{4}h + \frac{25}{4} 
\} \subset \Sing(\omega)$.
However, if $d=0$, then these cannot be integers simultaneously. 
Likewise, considering the two positive roots 
$2\alpha_1 + 3\alpha_2 + 4\alpha_3 + 6\alpha_4 + 5\alpha_5 + 3\alpha_6 + 2\alpha_7 + \alpha_8$  and 
$2\alpha_1 + 3\alpha_2 + 4\alpha_3 + 6\alpha_4 + 5\alpha_5 + 4\alpha_6 + 2\alpha_7 + \alpha_8$  in $\Phi_5$, 
we check that 
$\{ 
\frac{2}{5}a + \frac{3}{5}b + \frac{4}{5}c + \frac{6}{5}d + e + \frac{3}{5}f + \frac{2}{5}g + \frac{1}{5}h + \frac{26}{5},
\frac{2}{5}a + \frac{3}{5}b + \frac{4}{5}c + \frac{6}{5}d + e + \frac{4}{5}f + \frac{2}{5}g + \frac{1}{5}h + \frac{27}{5}
\} \subset \Sing(\omega)$.
However, if $f=0$, then these cannot be integers simultaneously. 
Therefore, 
there are no $a,b,c,d,e,f,g,h$ such that $\Sing(\omega)=\{ 1, 2, \cdots, 104 \}.$
\end{proof}

\subsection{Homogeneous variety $E_8/P_6$}
The rational homogeneous variety $E_8/P_6 \subset \mathbb P^{2450239}$ has dimension 97 and Fano index 14. 

\begin{proposition}
There are no irreducible equivariant Ulrich bundles on 
$E_8/P_6$.
\end{proposition}

\begin{proof}
Suppose that an irreducible equivariant vector bundle $\mathcal E_{\omega}$ on $E_8/P_6$ is Ulrich. 
Then $f=0$ by Proposition~\ref{coefficient of weight}. 
Because the Lie algebra $\mathfrak e_8$ has a gradation of depth 4 associated to $\alpha_6$ and 
$\alpha_2 + \alpha_3 + 2\alpha_4 + 2\alpha_5 + 2\alpha_6 + \alpha_7$, 
$\alpha_1 + \alpha_2 + 2\alpha_3 + 3\alpha_4 + 3\alpha_5 + 3\alpha_6 + 2\alpha_7 + \alpha_8$, 
$2\alpha_1 + 3\alpha_2 + 4\alpha_3 + 6\alpha_4 + 5\alpha_5 + 4\alpha_6 + 2\alpha_7 + \alpha_8$ 
are minimal among the 30 roots in $\Phi_2$, the 16 roots in $\Phi_3$ and the 3 roots in $\Phi_4$, respectively,  
we know that the corresponding singular values 
$t = 1 + \sum_{i \neq 6} \frac{c_i (a_i +1)}{2}>4$ for all $\alpha \in \Phi_2$, 
$t = 1 + \sum_{i \neq 6} \frac{c_i (a_i +1)}{3}>5$ for all $\alpha \in \Phi_3$ and 
$t = 1 + \sum_{i \neq 6} \frac{c_i (a_i +1)}{4}>6$ for all $\alpha \in \Phi_4$. 
Hence either $e=0$ or $g=0$ by Proposition \ref{subminimum}.

Considering the two positive roots 
$2\alpha_1 + 3\alpha_2 + 4\alpha_3 + 6\alpha_4 + 4\alpha_5 + 3\alpha_6 + 2\alpha_7 + \alpha_8$ and 
$2\alpha_1 + 3\alpha_2 + 4\alpha_3 + 6\alpha_4 + 5\alpha_5 + 3\alpha_6 + 2\alpha_7 + \alpha_8$ in $\Phi_3$, 
we check that 
$\{ 
\frac{2}{3}a + b + \frac{4}{3}c + 2d + \frac{4}{3}e + f + \frac{2}{3}g + \frac{1}{3}h + \frac{25}{3}, 
\frac{2}{3}a + b + \frac{4}{3}c + 2d + \frac{5}{3}e + f + \frac{2}{3}g + \frac{1}{3}h + \frac{26}{3}
\} \subset \Sing(\omega)$.
However, if $e=0$, then these cannot be integers simultaneously. 
Likewise, considering the two positive roots 
$\alpha_1 + \alpha_2 + 2\alpha_3 + 3\alpha_4 + 2\alpha_5 + 2\alpha_6 + \alpha_7 + \alpha_8$  and 
$\alpha_1 + \alpha_2 + 2\alpha_3 + 3\alpha_4 + 2\alpha_5 + 2\alpha_6 + 2\alpha_7 + \alpha_8$  in $\Phi_2$, 
we check that 
$\{ 
\frac{1}{2}a + \frac{1}{2}b + c + \frac{3}{2}d + e + f + \frac{1}{2}g + \frac{1}{2}h + \frac{13}{2}, 
\frac{1}{2}a + \frac{1}{2}b + c + \frac{3}{2}d + e + f + g + \frac{1}{2}h + 7
\} \subset \Sing(\omega)$.
However, if $g=0$, then these cannot be integers simultaneously. 
Therefore, 
there are no $a,b,c,d,e,f,g,h$ such that $\Sing(\omega)=\{ 1, 2, \cdots, 97 \}.$
\end{proof}

\subsection{Homogeneous variety $E_8/P_7$}
The rational homogeneous variety $E_8/P_7 \subset \mathbb P^{30379}$ has dimension 83 and Fano index 19. 

\begin{proposition}
There are no irreducible equivariant Ulrich bundles on 
$E_8/P_7$.
\end{proposition}

\begin{proof}
Suppose that an irreducible equivariant vector bundle $\mathcal E_{\omega}$ on $E_8/P_7$ is Ulrich. 
Then $g=0$ by Proposition~\ref{coefficient of weight}. 
Because the Lie algebra $\mathfrak e_8$ has a gradation of depth 3 associated to $\alpha_7$ and 
$\alpha_2 + \alpha_3 + 2\alpha_4 + 2\alpha_5 + 2\alpha_6 + 2\alpha_7 + \alpha_8$, 
$2\alpha_1 + 3\alpha_2 + 4\alpha_3 + 6\alpha_4 + 5\alpha_5 + 4\alpha_6 + 3\alpha_7 + \alpha_8$ 
are minimal among the 27 roots in $\Phi_2$ and the two roots in $\Phi_3$, respectively,  
we know that the corresponding singular values 
$t = 1 + \sum_{i \neq 7} \frac{c_i (a_i +1)}{2}>5$ for all $\alpha \in \Phi_2$ and 
$t = 1 + \sum_{i \neq 7} \frac{c_i (a_i +1)}{3}>9$ for all $\alpha \in \Phi_3$. 
Hence either $f=0$ or $h=0$ by Proposition \ref{subminimum}.

Considering the two positive roots 
$2\alpha_1 + 3\alpha_2 + 4\alpha_3 + 6\alpha_4 + 5\alpha_5 + 3\alpha_6 + 2\alpha_7 + \alpha_8$ and 
$2\alpha_1 + 3\alpha_2 + 4\alpha_3 + 6\alpha_4 + 5\alpha_5 + 4\alpha_6 + 2\alpha_7 + \alpha_8$ in $\Phi_2$, 
we check that 
$\{ 
a + \frac{3}{2}b + 2c + 3d + \frac{5}{2}e + \frac{3}{2}f + g + \frac{1}{2}h + 13, 
a + \frac{3}{2}b + 2c + 3d + \frac{5}{2}e + 2f + g + \frac{1}{2}h + \frac{27}{2}
\} \subset \Sing(\omega)$.
However, if $f=0$, then these cannot be integers simultaneously. 
Likewise, considering the two positive roots 
$2\alpha_1 + 3\alpha_2 + 4\alpha_3 + 6\alpha_4 + 5\alpha_5 + 4\alpha_6 + 3\alpha_7 + \alpha_8$  and 
$2\alpha_1 + 3\alpha_2 + 4\alpha_3 + 6\alpha_4 + 5\alpha_5 + 4\alpha_6 + 3\alpha_7 + 2\alpha_8$  in $\Phi_3$, 
we check that 
$\{ 
\frac{2}{3}a + b + \frac{4}{3}c + 2d + \frac{5}{3}e + \frac{4}{3}f + g + \frac{1}{3}h + \frac{28}{3}, 
\frac{2}{3}a + b + \frac{4}{3}c + 2d + \frac{5}{3}e + \frac{4}{3}f + g + \frac{2}{3}h + \frac{29}{3}
\} \subset \Sing(\omega)$.
However, if $h=0$, then these cannot be integers simultaneously. 
Therefore, 
there are no $a,b,c,d,e,f,g,h$ such that $\Sing(\omega)=\{ 1, 2, \cdots, 83 \}.$
\end{proof}

\subsection{Homogeneous variety $E_8/P_8$}
The rational homogeneous variety $E_8/P_8 \subset \mathbb P^{247}$ has dimension 57 and Fano index 29. 
Since the fundamental $E_8$-module $V_{E_8}(\omega_8)$ is the adjoint representation $\mathfrak e_8$, $E_8/P_8$ is the adjoint variety of $E_8$.

\begin{proposition}
There are no irreducible equivariant Ulrich bundles on 
$E_8/P_8$.
\end{proposition}

\begin{proof}
Suppose that an irreducible equivariant vector bundle $\mathcal E_{\omega}$ on $E_8/P_8$ is Ulrich. 
Then $h=0$ by Proposition~\ref{coefficient of weight}.
Because the decomposition of $\mathfrak e_8$ associated to $\alpha_8$ is a contact gradation and the singular value attained at $\theta \in \Phi_2$ is greater than 2,
$g=0$ by Proposition \ref{subminimum}.  
Since $\alpha_6 + \alpha_7 + \alpha_8$ is minimal among all roots in $\Phi_1 \setminus \{ \alpha_8 , \alpha_7 + \alpha_8\}$, 
the corresponding singular value $f + g + h + 3$ is equal to 3, which implies that $f=0$.
Likewise, since $\alpha_5 + \alpha_6 + \alpha_7 + \alpha_8$ is minimal among all roots in $\Phi_1 \setminus \{ \alpha_8 , \alpha_7 + \alpha_8, \alpha_6 + \alpha_7 + \alpha_8 \}$, 
the corresponding singular value $e + f + g + h + 4$ is equal to 4, which implies that $e=0$.

Considering the two positive roots 
$\alpha_1 + \alpha_2 + 2\alpha_3 + 3\alpha_4 + 3\alpha_5 + 2\alpha_6 + \alpha_7 + \alpha_8$  and 
$\alpha_1 + \alpha_2 + 2\alpha_3 + 3\alpha_4 + 2\alpha_5 + 2\alpha_6 + 2\alpha_7 + \alpha_8$ in $\Phi_1$, 
we check that 
$\{ 
a + b + 2c + 3d + 2e + 2f + 2g + h + 14,
a + b + 2c + 3d + 3e + 2f + g + h + 14
\} \subset \Sing(\omega)$.
However, if $e=f=g=h=0$, then these cannot be different. 
Therefore, 
there are no $a,b,c,d,e,f,g,h$ such that $\Sing(\omega)=\{ 1, 2, \cdots, 57 \}.$
\end{proof}

\vskip 1em

\noindent
\textbf{Acknowledgements}.
The authors would like to thank Rosa Maria Mir\'{o}-Roig and Laurent Manivel for useful discussions and comments. 
Also, they are grateful to the referee for several helpful suggestions to improve the first draft. 

\vskip 2em

\noindent
\textbf{Funding}.
This work was supported by the Institute for Basic Science (IBS-R003-Y1). 

\vskip 3em


\end{document}